\documentclass[amssymb,12pt]{article}
\usepackage{amsmath,amsthm}
\usepackage{graphicx,
}
\usepackage[psamsfonts]{amssymb}
\usepackage{amscd}

\title{Local density of diffeomorphisms \\ with large centralizers\\
\ \\
Densit\'e locale des diff\'eomorphismes \\ ayant un gros centralisateur}
\author{C. Bonatti, S. Crovisier, G.M. Vago and A. Wilkinson}

   \def\DD{{\mathbb D}}

 \def\NN{{\mathbb N}}  
\def\QQ{{\mathbb Q}} \def\RR{{\mathbb R}}  \def\TT{{\mathbb T}}
   
 \def\ZZ{{\mathbb Z}}

\def\Si{\Sigma}
\def\La{\Lambda}
\def\De{\Delta}
\def\Om{\Omega}

\def\si{\sigma}

  \def\cI{{\cal I}} \def\cO{{\cal O}} \def\cU{{\cal U}}
\def\cD{{\cal D}}    \def\cV{{\cal V}}
    
  \def\cL{{\cal L}}

\newtheorem*{mainthm}{Main theorem}
\newtheorem{ques}{Question}
\newtheorem{theo}{Theorem}
\newtheorem{prop}{Proposition}[section]
\newtheorem{lemma}[prop]{Lemma}
\newtheorem{lemm}[prop]{Lemma}
\newtheorem{coro}[prop]{Corollary}
\newtheorem{defi}[prop]{Definition}
\newtheorem{rema}[prop]{Remark}
\newtheorem{conj}{Conjecture}
\newtheorem*{clai}{Claim}

\def\Diff{\hbox{Diff} }
\def\Symp{\hbox{Symp} }

\def\title{\em}

\def\bar{\overline}

\def\T{\mathcal{T}}

\def\transverse{\,\raise2pt\hbox to1em{\hfil$\top$\hfil}\hskip -1em \hbox to1em{\hfil$\cap$\hfil}\,}
\newcommand\R{\mbox{\bf R}}

\newlength{\figboxwidth} 

\begin{document}

\maketitle
\begin{abstract}
Given any compact manifold $M$,  we construct a non-empty open
subset $\cO$ of the space $\Diff^1(M)$ of $C^1$-diffeomorphisms and
a dense subset $\cD\subset \cO$ such that the centralizer of every
diffeomorphism in $\cD$ is uncountable, hence non-trivial. \\
\ \\
{\bf R\'esum\'e :} Pour toute vari\'et\'e $M$ compacte, de dimension
quelconque, nous construisons une partie $\cO \subset \Diff^1(M)$ non
vide, ouverte dans l'espace $\Diff^1(M)$ des $C^1$-diff\'eomorphismes
de $M$, et un sous-ensemble $\cD \subset \cO$ dense en $\cO$,
constitu\'e de diff\'eomorphismes dont le centralisateur est non
d\'enombrable, donc non trivial. 
\ \\
{\bf  Mots cl\'e :} centralisateur trivial, sym\'etries triviales,
invariant de Mather.
\end{abstract}

\section*{Introduction}

The {\em centralizer} of a $C^r$ diffeomorphism $f\in \Diff^r(M)$ is the group of diffeomorphisms commuting with $f$:
 $$C(f):=\{g\in \Diff^r(M): fg=gf\}.$$

The centralizer $C(f)$ always contains the group $<f>$ of all the
powers of $f$. For this reason,  we say that $f$ has a {\em trivial
centralizer} if $C(f) = <f>$.  If $f$ is the time one map of a $C^r$
vector field $X$, then $C(f)$ contains the flow of $X$
and hence contains a subgroup diffeomorphic to $\RR$ (or $S^1=\RR/\ZZ$
if $f$ is periodic).

The elements of $C(f)$ are transformations of $M$ which preserve the dynamics of $f$: in that sense they are the symmetries of $f$.
How large is, in general, this symmetry group?

\begin{itemize}
\item
On one hand, the structure on $M$ given by a diffeomorphism is very flexible, so that one might expect that any symmetry could be broken
by a small perturbation of the diffeomorphism.
\item On the other hand, the symmetries are sought in the very large group
$\Diff^r(M)$, which makes the problem harder. For example, one
can easily show that the group $C^0(f)$ of  homeomorphisms
commuting with a Morse-Smale diffeomorphism $f$ is always
uncountable.
\end{itemize}
Nevertheless it is natural to guess that general diffeomorphisms
have no non-trivial smooth symmetries. Making this intuition
explicit, Smale asked the following:

\begin{ques}[\cite{Sm1,Sm2}]\label{c=main} Let
$\T^r(M)\subset \Diff^{~r}(M), r\geq 1$ denote the set of
$C^r$ diffeomorphisms of a compact manifold $M$ with trivial centralizer.

\begin{enumerate}
\item Is $\T^r(M)$ dense in $\Diff^{~r}(M)$?
\item Is $\T^r(M)$ residual in $\Diff^{~r}(M)$? That is, does it contain the intersection of countably many dense open subsets?
\item Is $\T^r(M)$ a dense open subset of $\Diff^{~r}(M)$?
\end{enumerate}
\end{ques}

We think it's natural to reformulate the third part of Smale's
question as:
\begin{ques}\label{q.opendense} Does $\T^r(M)$ contain a dense and open subset of $\Diff^{~r}(M)$?
\end{ques}

This question has many parameters, the most obvious being the
regularity $r$ of the diffeomorphisms and the dimension $\dim(M)$ of
the manifold. The question has been answered in varying degrees of generality
for specific parameters. For instance, Kopell \cite{Ko} proved that
$\Diff^r(S^1)$, $r\geq 2$, contains a dense and open subset of
diffeomorphisms with trivial centralizers. Many authors subsequently gave
partial answers in higher dimension (see \cite{BCW1} for an attempt to list these
partial results).

\vskip 5mm

The present paper and \cite{BCW2} together give a complete answer to Smale's problem for $r=1$. More precisely:
\begin{itemize}
\item \cite{BCW2} proves that $C^1$-generic diffeomorphisms have a trivial centralizer
\footnote{\cite{BCW1} is an announcement which gives the structure of the detailed proof
written in \cite{BCW2}.}, giving a positive answer to the first two
parts of Smale question. \cite{BCW3} shows that $C^1$-generic
conservative (volume preserving or symplectic) diffeomorphisms have
a trivial centralizer.
\item
In this paper, we answer in the negative  (for $r=1$) the third part of Smale's question (and to Question~\ref{q.opendense}) on any compact manifold.
\end{itemize}
\begin{mainthm}\label{t.main} Given any compact manifold $M$, there is a non-empty open subset $\cO\subset \Diff^1(M)$ and a dense subset
$\cD\subset \cO$ such that every $f\in \cD$ is $C^\infty$ and its $C^\infty$-centralizer $C^\infty(f)$ is uncountable (hence not trivial).
\end{mainthm}
We will see below (see Theorem~\ref{t.symp}) that this statement also holds for symplectic diffeomorphisms on a symplectic manifold.

The uniform presentation of this result (\emph{Given any compact
manifold,\dots}) hides very different situations, arguments and
results according to the dimension: namely, whether $\dim(M)<3$ or $\dim(M)\geq 3$. We
discuss this breakdown of the results below.

Our paper also deals with the question of how large is the class of
diffeomorphisms that can be embedded in a flow. This is a natural
question, since the studies of the dynamical systems defined either
by diffeomorphisms or by vector fields are in fact closely related.
In the paper \cite{Pa} titled \emph{Vector fields generate few
diffeomorphisms}, Palis proved that $C^1$-generic diffeomorphisms
cannot be embedded in a flow. Our results somehow counterbalance
Palis' statement: diffeomorphisms that are the time one map of a
flow are $C^1$-locally dense in dimension $1$ and $2$.

\begin{theo}\label{t=main2}
There is a dense subset $\cD\subset\Diff^{~1}(S^1)$ such that every $f\in\cD$
commutes with  the flow of  a $C^\infty$ Morse-Smale vector field $X$.
More precisely, $f$ is Morse-Smale and $f^q$ is the time one map of the flow of $X$,
where $q=2$, if $f$ is orientation reversing, and $q$ is the period of the periodic orbits of $f$ otherwise.
Furthermore, $C(f)$ is isomorphic to $\RR\times (\ZZ/q\ZZ)$.
\end{theo}
(Section~\ref{ss.questioncircle} presents open questions on
centralizers of diffeomorphisms in $\Diff^1(S^1)$.)

As for compact surfaces, the case of the sphere is specific because
of the existence north pole-south pole diffeomorphisms. The symmetries of these dynamics
allow us to get a centralizer isomorphic to $S^1\times \RR$ and
this is one of the reasons why we present this case separately.

Another specific feature of the north-south diffeomorphisms of the sphere
is that for these maps it is possible to generalize the so-called {\em Mather
invariant}, introduced by Mather in the one-dimensional case.
Such an invariant plays a fundamental role in our constructions: the
Mather invariant of a diffeomorphism $f$ is trivial if and only if
$f$ can be perturbed to become the time one map of a vector field.

\begin{theo}\label{t=sphere} Let $\cO\subset \Diff^{~1}(S^2)$ denote the (open)
subset of Morse-Smale diffeomorphisms $f$ such that the nonwandering
set $\Om(f)$ consists of two fixed points, one source $N_f$ and one
sink $S_f$, such that the derivatives $D_{N_f}f$ and $D_{S_f}f$ have
each a complex (non-real) eigenvalue.

Then there is a dense subset $\cD\subset\cO$ such that every
$f\in\cD$ is the time one map of a Morse-Smale $C^\infty$-vector
field. Furthermore $C(f)$ is   isomorphic to $\RR\times S^1$.
\end{theo}

Theorem~\ref{t=sphere} is a bridge between the one-dimensional case
and the general two-dimensional case.
One the one hand, north pole-south pole dynamics on the sphere and Morse-Smale
dynamics on the circle share the Mather invariant; on the other hand, other features of
these dynamics on the sphere occur in simple dynamics on a general compact surface.  The
general case is solved by a combination of the arguments used for
the sphere in a neighborhood of the sinks and the sources,
together with an analysis of the local situation in a neighborhood
of the saddles.

\begin{theo}\label{t.surface} Let $S$ be a  connected closed surface. Let $\cO\subset \Diff^{~1}(S)$ be the set
of Morse-Smale diffeomorphisms $f$ such that:
\begin{itemize}
\item any periodic point is a (hyperbolic) fixed point,
\item $f$ has at least one hyperbolic saddle point,
\item for any hyperbolic saddle $x$, every eigenvalue of $Df(x)$ is positive,
\item for any sink or source $x$, the derivative $Df(x)$ has a complex (non-real) eigenvalue,
\item there are no heteroclinic orbits: if $x\neq y$ are saddle points then $W^s(x)\cap W^u(y)=\emptyset$.
\end{itemize}
Then $\cO$ is a non-empty open subset of $\Diff^{~1}(S)$ and there
is a dense subset $\cD\subset \cO$ such that every $f\in\cD$ is the
time one map of a Morse-Smale $C^\infty$-vector field. Furthermore,
$C(f)$ is the flow of this vector field, hence isomorphic to~$\RR$.
\end{theo}

An important hypothesis in Theorems~\ref{t=sphere}
and~\ref{t.surface} (which holds trivially in Theorem~\ref{t=main2})
is that the derivative at each sink and source is
conjugate to the composition of a homothety with a (non-trivial)
rotation. This condition is open in dimension $2$, but is nowhere dense
in higher dimension. This explains why we are not able to obtain the
local density of the embeddability in a flow in higher dimension,
and naturally leads to the following question:

\begin{ques} Let $M$ be a compact manifold of dimension $d\geq 3$.
Denote by $\cO$ the $C^1$-interior of the $C^1$-closure of the set
of diffeomorphisms which are the time one map of a flow. Is $\cO$ empty?
\end{ques}

In low dimension we find large centralizers among the simplest
dynamical systems (the Morse-Smale systems). By contrast, in
higher dimension we will use $C^1$-open subsets of {\em wild
diffeomorphisms} to obtain periodic islands  where the return
map is the identity map.  The resulting large centralizers for these wild diffeomorphisms are
completely different. In low dimension, we embed the diffeomorphisms in a flow,
and the centralizer is precisely the flow;
hence all the diffeomorphisms in the centralizer have the same
dynamics. In higher dimension, the diffeomorphisms we exhibit  in
the centralizer will be equal to the identity map everywhere but in
the islands, in restriction to which they can be anything. This explains
our result:

\begin{theo}\label{t.wild} Given any compact manifold $M$ of dimension $d\geq 3$, there is a non-empty open subset $\cO\subset\Diff^{~1}(M)$
and a dense part $\cD\subset \cO$ such that every $f\in\cD$ has  non-trivial centralizer.

More precisely, for $f\in \cD$ the centralizer $C(f)$ contains a
subgroup isomorphic to the group $\Diff^{~1}(\RR^d,\RR^d\setminus \DD^d)$ of diffeomorphisms
of $\RR^d$ which are equal to the identity map outside the unit disc
$\DD^d$.
\end{theo}

The large centralizer we build for a diffeomorphism in Theorem~\ref{t.wild} consists of
diffeomorphisms which have a very small support, and which are therefore $C^0$-close to the identity.
It would be interesting to know if this
is always the case. Let us formalize this question:

\begin{ques} Let $M$ be a compact manifold with $\dim(M)\geq 3$ and $\varepsilon>0$.
Let $\cO_\varepsilon\subset \Diff^1(M)$ be the set of
diffeomorphisms $f$ such that, for every $g\in C(f)$ there exists
$n\in\ZZ$ such that $g\circ f^n$ is $\varepsilon$-close to the identity map
for the $C^0$-distance. Does $\cO_\varepsilon$ contain a dense open
subset of $\Diff^1(M)$ for every $\varepsilon$?
\end{ques}

If for non-conservative diffeomorphisms the existence of periodic
islands depends on wild dynamics, the same islands appear in a more natural way
for symplectic diffeomorphisms in a neighborhood of totally elliptic
points. In order to state precisely this last  result we need some
notations. Let $M$ be a compact manifold with even dimension $\dim
(M)=2d$. If $M$ carries a symplectic form $\omega$, then we denote by
$\hbox{Symp}^1_{\omega}(M)$ the space of $C^1$ diffeomorphisms of
$M$ that preserve $\omega$ (these diffeomorphisms are called
\emph{symplectomorphisms}).

Recall that a periodic point $x$ of period $n$  of a
symplectomorphism $f$ is called \emph{totally elliptic} if all the
eigenvalues of $Df^n(x)$ have modulus equal to $1$. If $e^{i\alpha}$ is
an eigenvalue of $x$ then $e^{-i\alpha}$ is also an eigenvalue.
Assume that $0<\alpha_1<\cdots<\alpha_{d}<\pi$ are the absolute
values of the argument of the eigenvalues of $x$. Then $x$ is
\emph{$C^1$-robustly totally elliptic}: every symplectomorphism $g$ that is
$C^1$-close to $f$ has a totally elliptic point $x_g$ of period $n$
close to $x$.\footnote{This uses the fact that if $\lambda$ is an eigenvalue of a symplectic matrix, then $\bar \lambda,\frac1\lambda,\frac1{\bar\lambda}$ are also eigenvalues.}

\begin{theo}\label{t.symp}Let $(M,\omega)$ be a symplectic manifold and let $\cO\subset \hbox{Symp}^1_{\omega}(M)$ denote the non-empty
open subset of the symplectomorphisms having a $C^1$-robust totally elliptic periodic point.

Then there is a dense part $\cD\subset \cO$ such that such that every $f\in\cD$ has a non-trivial centralizer.
More precisely, for $f\in \cD$ the symplectic centralizer $C_\omega(f)$ contains a subgroup isomorphic to the group $\Symp^1_\omega(\RR^{2d},\RR^{2d}\setminus \DD^{2d})$  of symplectomorphisms of $\RR^{2d}$ equal to the identity map on
the complement of the
unit disc $\DD^{2d}$.
\end{theo}

We remark that, according to \cite{ABC}, $C^1$-generic
symplectomorphisms of connected manifolds are transitive, and the manifold is a unique
homoclinic class. In particular, there is a $C^1$-dense and open
subset of symplectomorphisms having a hyperbolic periodic point with
some transverse homoclinic intersection. Such a transverse homoclinic
intersection prevents the diffeomorphism from being embedded in a
flow: thus non-embeddability in a flow is satisfied on a $C^1$-open
and dense subset of $\Symp^1_\omega(M)$. The same argument holds for
volume preserving diffeomorphisms, showing that, if $\mu$ is a
smooth volume form on a manifold $M$ of dimension $\dim(M)\geq 2$,
then the non embeddability in a flow is satisfied on a $C^1$-open
and dense subset of $\Diff^1_\mu(M)$. However we have not been able
to extend our main theorem for volume preserving
diffeomorphisms:

\begin{ques}Let $M$ be a closed manifold endowed with a smooth volume form $\mu$. Does there exist a non-empty open subset
$\cO\subset \Diff^1_{\mu}(M)$ and a dense subset $\cD\subset \cO$
such that for every $f\in \cD$ the centralizer $C_{\mu}(f)$ is not
trivial ?
\end{ques}

\section{Local embeddability in a flow on the circle}\label{s.flow}
\subsection{Preparation of diffeomorphisms of the circle}
The following proposition summarizes some very classical properties of diffeomorphisms of the circle.

\begin{prop}\label{p.reduction}
Let $\cD_0\subset \Diff^{~1}(S^1)$ be the set of diffeomorphisms $f$ satisfying the following properties:
\begin{itemize}
\item $f$ is a $C^\infty$ Morse-Smale diffeomorphism (i.e. the non wandering set consists of
finitely many hyperbolic periodic points, alternately attracting or
repelling);
\item for every periodic point $x\in Per(f)$, there is a neighborhood $U_x$ of $x$ such that  the restriction  $f|_{U_x}\colon U_x\to f(U_x)$
is  an affine map (for the natural affine structure on
$S^1=\RR/\ZZ$);
\item if $x,y\in Per(f)$ are periodic points with distinct orbits,
then $Df^q(x) \ne Df^q(y)$, where $q$ is the period of $x$.
\end{itemize}
Then $\cD_0$ is dense in $\Diff^1(S^1)$.
\end{prop}

For $\alpha>1$ and $\beta\in (0,1)$,
we introduce the set $D_{\alpha,\beta}$ of orientation-preserving $C^\infty$ diffeomorphisms
of the interval $[0,1]$ with the following properties:
\begin{itemize}
\item $\{0,1\}$ is the set of fixed points of $f$, and $f(x)>x$ for $x\in(0,1)$;
\item $f(x)=\alpha x$ for small $x$ and $f(x)=1+\beta(x-1)$ for $x$ close to $1$.
\end{itemize}

Applying Proposition~\ref{p.reduction}, we will prove Theorem~\ref{t=main2}
by working inside the space $D_{\alpha, \beta}$.

\subsection{The Mather invariant}
We recall here a construction introduced by J. Mather~\cite{mather} which associates to any diffeomorphism
$f\in D_{\alpha,\beta}$ a class of diffeomorphisms of $S^1$.

Let us fix $\alpha>1$ and $\beta\in (0,1)$ and introduce a $C^\infty$ orientation preserving diffeomorphism
$\varphi\colon (0,1)\to \R$ such that $\varphi(x)= {\ln x}/{\ln (\alpha)}$ for $x$ small, and
$\varphi(x)= {\ln (1-x)}/{\ln(\beta)}$  for $x$ close to $1$.  Observe that there exists a constant $K_0>0$ such that
$\varphi^{-1}(x)= \exp(\ln(\alpha)\, x )$ for $x<-K_0$ and  $\varphi^{-1}(x)= 1-\exp(\ln(\beta)\, x)$, for $x>K_0$.

For any $f\in D_{\alpha,\beta}$  the conjugated diffeomorphism $\theta_f=\varphi\circ f\circ \varphi^{-1}$
of $\RR$ satisfies $\theta_f(x)>x$ for all $x$; furthermore, $\theta_f(x)$ agrees with $x+1$
when $|x|$ is larger than a constant $K_f>K_0$.

The space $\RR/\theta_f$ of the orbits of $\theta_f$ is a smooth circle $S_f$ which has
two natural identifications with the (affine) circle $S^1=\RR/\ZZ$:
two points $x,y\in (-\infty,-K_f]$ (resp. $x,y\in [K_f,+\infty)$) are in the  same orbit for $\theta_f$
if and only if they differ by an integer. This leads to two diffeomorphisms
$\pi_+\colon S_f\to S^1$ and $\pi_-\colon S_f\to S^1$, respectively.
Let  $\De_{f,\varphi}=\pi_+\circ \pi_-^{-1}\colon S^1\to S^1$.

\begin{prop} \label{p.matherdef}
The diffeomorphism $f$ is the time one map of a $C^1$-vector field
if and only if $\De_{f,\varphi}$ is a rotation.
\end{prop}
\begin{proof} Note that $f\in D_{\alpha,\beta}$ agrees with the time one map of the vector field
$X^-= \ln(\alpha)\,x \, \frac\partial{\partial x}$ in a neighborhood
of $0$ and with $X^+= \ln(\beta)\, (x-1) \, \frac\partial{\partial
x}$ in a neighborhood of $1$. Furthermore, if $f$ is the time one
map of a $C^1$-vector field $X$ on $[0,1]$, then $X=X^-$ in a
neighborhood of $0$ and $X=X^+$ in a neighborhood of $1$. The
hypothesis on $\varphi$ implies
$\varphi_*(X_-)=\frac\partial{\partial x}$ on some interval
$(-\infty,L_-)$ and $\varphi_*(X_+)=\frac\partial{\partial x}$ on an
interval $(L_+,\infty)$.

Assume that $\De_{f,\varphi}$ is a rotation. Then we define a vector
field  $Y$ on $\RR$ as follows: consider $n>0$ such that
$\theta_f^n(x)> K_f$. Now let $Y(x)=\left(D_x\theta_f^n \right)^{-1}(\frac\partial{\partial x})$. This vector does not depend of $n$
(because $\theta_f$ is the translation $t\mapsto t+1$ for $t\geq K_f$).

\begin{clai} if $x<-K_f$ then $ Y(x)=\frac\partial{\partial x}$.
\end{clai}
\begin{proof} Consider the natural projection $\pi_f\colon \RR\to S_f$ that maps each point to its orbit for $\theta_f$.
Since $Y$ is invariant under $\theta_f$, the vector field $(\pi_f)_*(Y)$ is well-defined. Since on $(K_f,+\infty)$
the vector $Y(x)$ is equal to $\frac\partial{\partial x}$, the map $\pi_+\circ \pi_f$
agrees with the natural projection $\RR\to S^1$, and we have
$(\pi_+\circ\pi_f)_*(Y(x))=\frac\partial{\partial x}$.
As $\De_{f,\varphi}$ is a rotation, and as the rotations preserve the vector field $\frac\partial{\partial x}$,
we obtain that
$(\pi_-\circ\pi_f)_*(Y(x))=(\De_{f,\varphi}^{-1}\circ \pi_+\circ\pi_f)_*(Y(x))=\frac\partial{\partial x}$.
As $\theta_f$ agrees with the translation $t\mapsto t+1$ on $(-\infty,-K_f]$, the projection $\pi_-\circ \pi_f$
agrees on $(-\infty,-K_f]$ with the natural projection $\RR\mapsto\RR/\ZZ$.
Hence $(\pi_-\circ\pi_f)_*(Y(x))=\frac\partial{\partial x}$ implies $Y(x)=\frac\partial{\partial x}$.
\end{proof}

Notice that, by construction, the vector field $Y$ is invariant
under $\theta_f$; furthermore, $\theta_f$ is the time one map of
$Y$: this is true on  a neighborhood of $\pm\infty$, and extends on
$\RR$ because $Y$ is $\theta_f$-invariant.

Now,  the vector field $X=\varphi^{-1}_*(Y)$, defined on $(0,1)$, agrees with $X_-$ and $X_+$ in a neighborhood
of $0$ and $1$, respectively, hence induces a smooth vector field on $[0,1]$.
Finally, $f$ is the time one map of $X$.

Conversely, if $f$ is the time one map of a $C^1$-vector field $X$ on $[0,1]$ then $\theta_f$ is the time one map
of  the vector field $Y=\varphi_*(X)$, which agrees with $\partial/\partial x$ in the neighborhood
of $\pm\infty$ (because $X$ agrees with $X_-$ and $X_+$ in a neighborhood of $0$ and $1$, respectively).
Hence the projections $(\pi_-\circ\pi_f)_*(Y)$ and $(\pi_+\circ\pi_f )_*(Y)$ are both equal
to the vector field $\partial/\partial x$ on $S^1$. This implies that
$(\De_{f,\varphi})_* (\partial/\partial x)=\partial/\partial x$, which
implies that $\De_{f,\varphi}$ is a rotation.
\end{proof}

\begin{rema}
The function $\De_{f,\varphi}$ defined here seems to depend on the choice of $\varphi$.
There is a more intrinsic way to define the diffeomorphism  $\De_{f,\varphi}$ ``up to composition
with rotations:''

The vector fields $X_-$ and $X_+$ defined in a neighborhood of $0$ and $1$, respectively,
are the unique vector fields such that  $f$ is the time one map of the corresponding flows,
in the neighborhood of $0$ and $1$, respectively. Each of these vector fields induces
a parametrization of the orbit space $(0,1)/f=S_f$, that is, up to the choice of an origin,
a diffeomorphism $\pi_f^\pm\colon S_f\to S^1$. The change of parametrization $\pi_f^+\circ (\pi_f^-)^{-1}$
is well-defined, up to the choice of an origin of the circle, i.e. up to composition,
at the right and at the left, by rotations. This class of maps is
called the {\em Mather invariant} of $f$.
\end{rema}

\subsection{Forcing the Mather invariant to vanish}\label{ss.contre-exemple}
In this subsection we prove the following result.
\begin{prop}\label{p.contre-exemple} Let $f$ be a diffeomorphism in $D_{\alpha,\beta}$.
Then any $C^1$-neighborhood $\cU$ of $f$ in $\Diff^1([0,1])$
contains a diffeomorphism $g$ such that $g=f$ in a neighborhood of
$\{0,1\}$ and $g$ is the time one map of a $C^\infty$-vector field
on $[0,1]$.
\end{prop}
We retain the notation from the previous subsection.
Fix $f\in D_{\alpha,\beta}$ and $K_f>0$ such that $\theta_f=\varphi \circ f\circ \varphi^{-1}$
agrees with $x\mapsto x+1$ on $(-\infty,-K_f]\cup[K_f,+\infty)$.

Given a diffeomorphism $h\colon \RR\to\RR$, the \emph{support of $h$}, denoted by $supp(h)$
is the closure  of the set of points $x$ such that $h(x)\neq x$.
\begin{lemm} \label{l.composition} Consider a number $a>K_f$ and a diffeomorphism
$\tilde \psi\colon\RR\to\RR$  whose support is contained in $(a,a+1)$.
Let $h$ denote the  diffeomorphism $\varphi^{-1}\circ \tilde
\psi\circ\varphi$, and let
$\psi$ denote the diffeomorphism of $S^1\simeq [a,a+1]/a\sim a+1$ induced by $\tilde \psi$.

Then the diffeomorphism $g= f\circ h$ belongs to $D_{\alpha,\beta}$, and
$\De_{g,\varphi}=\psi \circ\De_{f,\varphi}$.
\end{lemm}
\begin{proof} The diffeomorphism $g$ agrees with $f$ in neighborhoods of $0$ and $1$, which  proves that
$g\in D_{\alpha,\beta}$. Furthermore, by construction, one may choose $K_g=a+1$.

For $x<-a$, there is a (unique) integer such that $\theta_f^n(x)=\theta_g^n(x)\in[a,a+1)$,
and by construction of $\De_{f,\varphi}$,  the projection of $\theta_f^n(x)$ on $S^1$ is $\De_{f,\varphi(x)}$.
Now the projection on $S^1$ of $\theta_g^{n+1}(x)=\theta_f\circ \tilde \psi\circ\theta_f^n(x)$
is $\psi\circ\De_{f,\varphi}(x)$, by construction.
As $\theta_g=\theta_f=y\mapsto y+1$ for $y\geq a+1$, one gets that the projection on $S^1$ of $\theta_g^{n+k}(x)$
is $\psi\circ\De_{f,\varphi}(x)$, for all $k>0$; hence $\De_{g,\varphi}=\psi\circ\De_{f,\varphi}$.
\end{proof}

Iterating the process described in Lemma~\ref{l.composition}, we obtain:

\begin{coro}\label{c.composition} Consider a finite sequence of numbers $a_i>K_f$, $i\in\{1,\dots,\ell\}$,
such that $a_{i+1}>a_i+1$ for all $i\in\{1,\dots,\ell-1\}$. For every $i\in\{1,\dots,\ell\}$, fix
a diffeomorphism $\tilde \psi_i\colon\RR\to\RR$ whose support is contained in $(a_i,a_i+1)$.
Let $h_i$ denote the diffeomorphism $\varphi^{-1}\circ \tilde \psi_i\circ\varphi$,
and let $\psi_i$ denote the diffeomorphism of $S^1$ induced by $\tilde \psi_i$.
(Note that the diffeomorphisms $h_i$ have disjoint support, so that they are pairwise commuting.)

Then the diffeomorphism $g= f\circ h_1\circ h_2\circ \cdots \circ h_\ell $ belongs to $D_{\alpha,\beta}$,
and we have: $$\De_{g,\varphi}=\psi_\ell\circ\cdots\circ\psi_1 \circ\De_{f,\varphi}.$$
\end{coro}

\begin{defi}\label{d.Theta}
Let $a\in\RR$, and let $\bar a$ be its projection on $S^1=\RR/\ZZ$. Given a diffeomorphism $\psi\colon S^1\to S^1$
with support in $S^1\setminus\{\bar a\}$ we call \emph{the lift of $\psi$  in $(a,a+1)$}
the diffeomorphism $\tilde \psi_a\colon\RR\to \RR$ with support in $(a,a+1)$ such that for any
$x\in(a,a+1)$ the image $\psi_a(x)$ is the point of $(a,a+1)$ which projects to $\psi(\bar x)$
where $\bar x$ is the projection of $x$.

We denote by $\Theta_a(\psi)$ the diffeomorphism of $[0,1]$ whose expression in $(0,1)$
is $\Theta_a(\psi)=\varphi^{-1}\circ \psi_a\circ\varphi$.
\end{defi}

\begin{lemm}\label{l.voisinages} For any $C^1$-neighborhood $\cU$ of $f$ there is a neighborhood $\cV$
of $Id_{S^1}\in\Diff^1(S^1)$ with the following property:

Given any finite sequence $a_i>K_f$, $i\in\{1,\dots,\ell\}$, such that $a_{i+1}>a_i+1$
for all $i\in\{1,\dots,\ell-1\}$, we denote by $\bar a_i$ the projection of $a_i$ on $S^1$.
For any~$i$, let $\psi_i\in\cV$ be a diffeomorphism of $S^1$ with support in $S^1\setminus\{a_i\}$.
Then the diffeomorphism $g= f\circ \Theta_{a_1}(\psi_1)\circ \cdots \Theta_{a_\ell}(\psi_\ell)$
belongs to $\cU$.
\end{lemm}
\begin{proof} We fix a neighborhood $\cU_0$ if the identity map of $[0,1]$ such that,
if $g_1,\dots,g_n\in\cU_0$ and if the support of the $g_i$ are pairwise disjoint, then
$ f\circ g_1\circ h_2\circ \cdots g_n $ belongs to $\cU$.  Now the lemma is a direct consequence
of Lemma~\ref{l.voisinages2} below.
\end{proof}
\begin{lemm}\label{l.voisinages2}
For any $C^1$-neighborhood $\cU_0$ of $f$ there is a neighborhood $\cV$ of $Id_{S^1}\in\Diff^1(S^1)$
with the following property:

Consider any  $a>K_f$, its  projection $\bar a$ on $S^1$ and any diffeomorphism  $\psi\in\cV$
with support in $S^1\setminus\{a\}$. Then the diffeomorphism $\Theta_{a}(\psi)$ belongs to $\cU_0$.
\end{lemm}
\begin{proof} Notice that there exists $\varepsilon>0$ such  that $\cU_0$ contains any diffeomorphism $h$ of $[0,1]$
with $\sup_{x\in[0,1]}|D_xh -1| <\varepsilon$.

Now consider $a>K_f$ and an integer $n>0$. Then for any diffeomorphism $\psi$ of $S^1$ with support
in $S^1\setminus \{\bar a\}$, the lifts $\psi_{a}$ and $\psi_{a+n}$ are conjugated by the translation
$x\mapsto x+n$. As a consequence, $\Theta_{a+n}(\psi)$ is obtained from $\Theta_{a}(\psi)$ by the conjugacy
by the homothety of ratio  $\beta^n$.
As a consequence we obtain that
$\sup_{x\in[0,1]}|D_x\Theta_{a+n}(\psi)-1|=\sup_{x\in[0,1]}|D_x\Theta_{a}(\psi)
-1|$.

Hence it suffices to prove the lemma for $a\in[K_f,K_f+1]$. This is a direct consequece of the facts
that the derivatives of $\varphi$ and of $\varphi^{-1}$ are bounded on
$\varphi^{-1}([K_f,K_f+2])$ and $[K_f,K_f+2]$ respectively, and that for any $\psi$ with support in
$S^1\setminus\{a\}$, one has:
$$\sup_{x\in[a,a+1]}|D_x\psi_a-1|=\sup_{x\in S^1}|D_x\psi-1|.$$
\end{proof}

Let us now recall a classical result which is the key point of our proof.

\begin{theo}[Fragmentation lemma]\label{t.decomposition}
Let $M$ be a closed Riemannian manifold, let $r>0$ and let $\cU$ be
a $C^1$-neighborhood of the identity map. Then for any smooth
diffeomorphism $f$ of $M$ isotopic to the identity, there exist
$k\geq 1$ and $g_1,\dots , g_k\in \cU$ such that $g_i=id$
 on the complement of a ball $B(x_i,r)$, and
$$f=g_1\circ\cdots\circ g_k.$$
\end{theo}
Here we use Theorem~\ref{t.decomposition} on the circle $S^1$, where
it is an easy consequence of the result, by M. Herman, that any
smooth diffeomorphism is the product of a rotation by a
diffeomorphism smoothly conjugate to a rotation. In
Section~\ref{ss.sphere}, we will also use
Theorem~\ref{t.decomposition} on the torus $T^2$ and on compact
surfaces.
\bigskip

\begin{proof}[\noindent Proof of Proposition~\ref{p.contre-exemple}]
Given a $C^1$-neighborhood $\cU$ of $f$, we choose a $C^1$-neighborhood $\cV$ of the identity map of $S^1$
given by Lemma~\ref{l.voisinages}. Using
Theorem~\ref{t.decomposition},
we can write $\De_{f,\varphi}$ as a finite
product
$\De_{f,\varphi}=\psi_1^{-1}\circ\cdots\circ\psi_\ell^{-1}$ such that
$\psi_i\in\cV$, and the support of $\psi_i$
is contained in an interval of length $\frac12$ in $S^1$ (and in
particular is not all of $S^1$).
Now we choose a finite sequence $a_i>K_k$ such that $a_{i+1}>a_i+1$, and such that the projection $\bar a_i$
does not belong to the support of $\psi_i$. Let $h_i=\Theta_{a_i}(\psi_i)$.

Applying Lemma~\ref{l.voisinages}, we obtain that the diffeomorphism
$$g=f\circ h_1\circ h_2\circ \cdots \circ h_\ell $$
belongs to $\cU$; applying Corollary~\ref{c.composition}, we get that
$$\De_{g,\varphi}=\psi_\ell\circ\cdots\circ\psi_1 \circ\De_{f,\varphi}= Id_{S^1}.$$
\end{proof}

\subsection{Existence of an invariant vector field}\label{sss=circle}
In this section we prove the first part of Theorem~\ref{t=main2}.
\begin{prop}
Let $\cD_0^*$ be the set of diffeomorphisms $f\in \cD_0$ which preserve
a $C^\infty$ Morse-Smale vector field $X$ and such that $f^q$ is the time one map
of $X$, where $q$ is the period of the connected components of $S^1\setminus Per(f)$.
Then $\cD_0^*$ is dense in $\Diff^1(S^1)$.
\end{prop}
\begin{proof}
By Proposition~\ref{p.reduction}, it is enough to prove that $\cD_0^*$ is dense in $\cD_0$.
Consider $f\in \cD_0$. The set $Per(f)$ is finite. Let $\cI$ be the set of segments joining two successive periodic points of $f$;
in other words, every element $I\in\cI$ is the closure of a connected component of $S^1\setminus \hbox{Per}(f)$.
Notice that $f$ induces a permutation on $\cI$. Furthermore, all the elements of $\cI$ have the same period
denoted by $q>0$, under this action (this period is equal to $2$ if $f$ reverses the orientation,
and is equal to the period of the periodic orbits in the orientation preserving case).

Now consider a segment $I\in \cI$. The endpoints of $I$ are the fixed points of the restriction $f^q|_I$;
moreover, one endpoint (denoted by $a$) is a repeller and the other (denoted by $b$) is an attractor.
Let $h_I\colon I\to [0,1]$ be the affine map such that $h_I(a)= 0$ and $h_I(b)=1$, and
let $\varphi_I\colon[0,1]\to [0,1]$ denote the diffeomorphism $h_I\circ f^q|_I\circ h_I^{-1}$.

According to Proposition~\ref{p.contre-exemple}, there is a sequence $(\psi_{I,n})_{n\in\NN}$, of diffeomorphisms
converging to $\varphi_I$ in the $C^1$-topology when $n\to+\infty$,
and a sequence $(Y_{I,n})_{n\in \NN}$ of $C^\infty$ vector fields on $[0,1]$
such that $\psi_{I,n}$ agrees with $\varphi_I$ in a small neighborhood of $\{0,1\}$
and is time one map of $Y_{I,n}$.
We define $g_{I,n}=h_I^{-1}\circ \psi_{I,n}\circ h_I$.
Notice that each diffeomorphism $g_{I,n}$ agrees with $f^q$ in
neighborhoods of the endpoints of $I$ and
converges to $f^q|_I$ when $n\to \infty$.

We now define a diffeomorphism $f_{I,n}$ of $S^1$ as follows:
$$f_{I,n}=
\begin{cases}
f & \hbox{ on } S^1\setminus f^{q-1}(I)\\
g_{I,n}\circ f^{-q+1}&\hbox{ on }f^{q-1}(I).
\end{cases} $$
This is a $C^\infty$ diffeomorphism since it agrees with $f$ in a neighborhood of the
periodic orbits. Moreover, $(f_{I,n})$ converges to  $f$ as $n$ approaches $+\infty$.

We denote by $X_{I,n}$ the vector field, defined on the orbit $\bigcup_0^{q-1}f^i(I)$
of the segment $I$, as follows:
\begin{itemize}
\item $X_{I,n}= (h_I^{-1})_*(Y_{I,n})$ on $I$;
\item for all $i\in\{1,\dots,q-1\}$ and all $x\in I$:
$$X_{I,n}(f^i(x))=f^i_*(X_{I,n}(x)).$$
\end{itemize}

Finally, we fix a family $I_1,\dots,I_\ell\subset \cI$ such that for $i\neq j$
the segments $I_i$ and $I_j$ have distinct orbits, and
conversely every orbit of segment in $\cI$ contains one of the $I_i$.

We denote by $f_n$ the diffeomorphism of $S^1$ coinciding with $f_{I_i,n}$ on the orbit of $I_i$
for all $i\in\{1,\dots, \ell\}$. This diffeomorphism is well-defined because all the  $f_{I_i,n}$
agree with $f$ in a small neighborhood of the periodic points (the endpoints of the segments in $\cI$).
We denote by $X_n$ the vector field on $S^1$ that agrees with $X_{I_i,n}$ on the orbit of $I_i$,
for all $i\in\{1,\dots, \ell\}$.

It is easy to see that $X_n$ is a smooth vector field on $S^1$, invariant under $f_n$, and
such that $f_n^q$ is the time one map of $X_n$: the unique difficulty
consists in checking the continuity
and smoothness of the vector field $X_n$ at the periodic points.
As $f_n$ is affine in the neighborhood
of the periodic orbits, it follows that, at both sides of a periodic point $x$, the vector field $X_n$
is the affine vector field vanishing at $x$ and whose eigenvalue at
$x$ is  $\ln D_xf$.
We hence have built a sequence $(f_n)$ in $\cD_0^*$ converging to $f$ in the $C^1$-topology, as required.
\end{proof}

\subsection{Centralizer of a diffeomorphism $f\in\cD_0^*$}\label{ss.centre-cercle}

To complete the proof of Theorem~\ref{t=main2}, it remains to
exhibit the centralizer $C(f)$, for $f\in\cD_0^*$.
Let $X$ be the smooth vector field
such that $f^q=X^1$. Denote by $\{x_i\}$ the set of periodic points of $f$
indexed in an increasing way for a cyclic order
(according to the natural  orientation of the circle $S^1=\RR/\ZZ$).
The flow $\{X^t\}_{t\in\RR}$ of $X$ is contained in the
centralizer $C(f)$. Let $h_0= f\circ X^{-\frac1q}$;
it belongs to $C(f)$, it acts on $Per(f)$ as $f$, and $h_0^q=id$.
The group $G_f$ generated by the flow of $X$ and by $h_0$
is isomorphic to $\RR\times \ZZ/q\ZZ$ and is contained in $C(f)$.
We will prove:
\begin{prop}\label{p.centre-cercle}
With the notation above, $C(f)$ is isomorphic to $\RR\times \ZZ/q\ZZ$ or $\RR\times \ZZ/q\ZZ\times \ZZ/2\ZZ$.
More precisely, $C(f)=G_f\simeq \RR\times \ZZ/q\ZZ$, except in the special case where $f$ preserves the orientation and has exactly two periodic orbits; in this special case, either $C(f)=G_f\simeq \RR\times \ZZ/q\ZZ$ or
$C(f)= G_f\times \ZZ/2\ZZ\simeq \RR\times \ZZ/q\ZZ\times \ZZ/2\ZZ$.
\end{prop}

Consider now $g\in C(f)$.  Let $Y^t=g\circ X^t\circ g^{-1}$. Then $\{Y^t\}_{t\in\RR}$ is a one parameter group of $C^1$-diffeomorphisms commuting
with $f$, and $f^q=Y^1$. One easily deduces that $Y^t$ fixes every (oriented) segment $[x_i,x_{i+1}]$. As $f^q$ is an affine contraction or dilation in the neighborhood of $x_i$ and $x_{i+1}$, one also easily deduces that
$Y^t=X^t$ for every $t$, so that $g$ commutes with the flow of  $X$. In particular, this shows that the group $G_f$ is contained in the center of $C(f)$
(i.e. every element of $G_f$ commutes with every element of $C(f)$).

The diffeomorphism $g$ preserves the (finite) set $Per(f)$, so that $Per(f)$ consists in
periodic orbits of $g$. Let $\ell>0$ be the smallest positive
integer such that $g^\ell$ preserves the orientation and has a fixed
point. This implies that every $x_i\in Per(f)$ is a fixed point of
$g^\ell$. As $f^q$ is an affine contraction or dilation in a
neighborhood of $x_i$, and $g^\ell$ commutes with $f^q$, it follows
that $g^\ell$ is an affine map in the neighborhood $x_i$: in other
words, $g^\ell$ agrees in the neighborhood of $x_i$ with the time-$t_i$
map $X^{t_i}$ of the flow of $X$. As a consequence, $g^\ell$ agrees
with $X^{t_i}$ on the basin (stable or unstable manifold) of $x_i$.
Since the basin of $x_i$ meets the basin of $x_{i+1}$, an
inductive argument shows that the time $t_i$ does not depend on $i$. This proves
that there exists a $t$ such that $g^\ell= X^t$. Let $h_g=g\circ X^{-\frac
t\ell}$. Then $h$ belongs to $C(f)$, induces the same
permutation of the periodic points of $f$ as $g$, and (as $g$
commutes with $X^{-\frac t\ell}$), we have $h_g^\ell=id$.

The proposition now follows from two remarks:
\begin{itemize}
\item Since the derivatives (at the period) of the periodic orbits $f\in\cD_0$
are pairwize distinct, any diffeomorphism $g\in C(f)$ preserves each periodic orbit of $f$.
\item An orientation preserving diffeomorphism of $S^1$ which is periodic and has a fixed point is the identity map.
\end{itemize}

We can now complete the proof:

\noindent{\bf {\underline{Case a:} $f$ reverses the orientation.}} Then $f$ has two fixed points.
Assume $g\in C(f)$ preserves the orientation; then so does $h_g$. By the first remark, $h_g$ fixes the fixed points of $f$. By the second remark,  $h_g=id$, and so $g=X^{\frac t\ell}\in G_f$.

If $g$ reverses the orientation, then $f^{-1}\circ g\in C(f)$ preserves the orientation: so there exists $r\in\RR$ such that $g=X^r\circ f\in G_f$.

This shows that $C(f)=G_f\simeq \RR\times \ZZ/2\ZZ$.

\noindent{\bf{\underline{Case b:} $f$ preserves the orientation and has at least $3$ periodic orbits.}}
Then the periodic orbits of $f$ are cyclically ordered. Since $g\in C(f)$ preserves every periodic orbit,
it preserves this order, and hence $g$ preserves the orientation.

Let $x\in Per(f)$. Then $h_g(x)$ belongs to the $f$-orbit of $x$, hence to the $h_0$ orbit of $x$. So there exists $r$ such that $h_g\circ h_0^r$ has a fixed point. As $h_0$ belongs to the center of $C(f)$, it commutes with $h_g$. As $h_0$ and $h_g$ are periodic we deduce that $h_g\circ h_0^r$ is periodic, hence is the identity map, because it is orientation-preserving and has a fixed point. So $h_g=h_0^r$. This proves that $g\in G_f$, and so $C(f)=G_f\simeq \RR\times \ZZ/q\ZZ$.

\noindent{\bf{\underline{Case c:} $f$ preserves the orientation and has exactly $2$ periodic orbits.}}
Notice that this case represents an open subset  $\cD_1\subset \cD_0^*$. We will prove that every $f\in\cD_1$ satisfies $C(f)=\RR\times \ZZ/q\ZZ\times \ZZ/2\ZZ$ or $C(f)=G_f\simeq \RR\times \ZZ/q\ZZ$ and that both behaviors are dense in $\cD_1$.
The argument in the previous case shows that every $g\in C(f)$ preserving the orientation belongs to $G_f$. The two possible behaviors depend on the existence of an orientation-reversing $g\in C(f)$. We will exhibit an invariant (of differentiable conjugacy) that vanishes if and only if $C(f)$ contains an orientation-reversing element.

The space of orbits of $h_0$ is a circle $S_0= S^1/ h_0$, and $f$ induces on that circle a diffeomorphism $f_0= f/h_0$ with exactly $2$ fixed points: one attractor and one repellor; moreover, there are smooth coordinates on $S_0\simeq \RR/\ZZ$ such that the fixed points of $f_0$ are a source at $0$  and a sink at $\frac12$, and such that $f_0$ is affine in the neighborhood of its fixed points. We fix $\varepsilon>0$ such that $f_0$ is affine on $[-\varepsilon, \varepsilon]$ and on  $[\frac12 -\varepsilon, \frac12+\varepsilon]$.

The vector field $X$ induces on $S_0$ a vector field $Y$, invariant
by $f_0$, and $f_0^q$ is the time $1$ map of $Y$. It follows that
$f_0=Y^{\frac 1q}$.

For $r,s\in(0,\varepsilon]$ let $t^+(r,s)$ and $t^-(r,s)$ be the time length of the $Y$-orbit from $r$ to $\frac12-s$ and from $-r$ to $\frac12 +s$.
One easily checks that $t^+(r,s)-t^-(r,s)$ does not depend on $r,s\in(0,\varepsilon]$; let $\theta_f=t^+(r,s)-t^-(r,s)$.

\begin{lemm} $C(f)$ contains an orientation-reversing element if and only if $\theta_f=0$.
\end{lemm}
\begin{proof} If $g\in C(f)$ is orientation-reversing, then $h_g$ projects on $S_0$ to a diffeomorphism commuting with $Y$ and agreeing with $r\mapsto -r$ on $[-\varepsilon,\varepsilon)$ and with $\frac 12-s\mapsto \frac 12+s$ on $[\frac12-\varepsilon, \frac12+\varepsilon]$.  This implies $t^-(r,s)=t^+(r,s)$. Conversely, if $\theta_f=0$,
then it is possible to build an orientation-reversing symmetry for the flow $Y$, agreeing with $r\mapsto -r$ on $[-\varepsilon,\varepsilon)$ and with $\frac 12-s\mapsto \frac 12+s$ on $[\frac12-\varepsilon, \frac12+\varepsilon]$. This symmetry lifts on $S^1$ to a element $g\in C(f)$.
\end{proof}

We conclude the proof of the proposition by proving
\begin{lemm}\label{l.theta} The subsets $\{f\in\cD_1,\theta_f=0\}$ and $\{f\in\cD_1,\theta_f\neq 0\}$ are both dense in $\cD_1$.
\end{lemm}
The proof is very analoguous to the argument that forces the Mather invariant to vanish. One proves that $C^1$-small perturbations can give arbitrary values of $\theta_f$.

If $\cD^*_1$ denotes the second set that appears in lemma~\ref{l.theta},
then Theorem~\ref{t=main2} is now proved with the set
$$\cD=(\cD_0^*\setminus \cD_1)\cup \cD_1^*.$$
\subsection{Open questions on the circle}\label{ss.questioncircle}

Theorem~\ref{t=main2} presents a dense subset of $\Diff^1(S^1)$ of
smooth diffeomorphisms whose centralizer is not trivial, and more
precisely is isomorphic to a group $\RR\times \ZZ/q\ZZ$, for some
integer $q$. However, it is known that $C^1$-diffeomorphisms of
$S^1$ may have very different centralizers. For instance, \cite{FF}
presents faithful actions on $S^1$ of any finitely
generated nilpotent (non-abelian) group $G$. This nilpotent group
$G$ has a non-trivial center $Z(G)$, and for every element $f\in
Z(G)$ of this center, the centralizer $C(f)$ contains the whole
group $G$, hence is not abelian.

\begin{conj} The set of diffeomorphisms $f\in\Diff^1(S^1)$ with a non-abelian centralizer is dense in $\Diff^1(S^1)$.
\end{conj}
It could be interesting to build examples of diffeomorphisms with irrational rotation number and having a non-abelian centralizer.

Theorem~\ref{t=main2} uses the fact that $C^1$-small perturbations
allow us to change arbitrarily the Mather invariant of any smooth
diffeomorphism of the interval $[0,1]$ whose fixed points are
precisely $0$ and $1$. This proves that, in the set
$D_{\alpha,\beta}$ we defined, every class of smooth conjugacy is
$C^1$-dense. This suggests another question. For $\alpha>1$ and
$0<\beta<1$, consider the set
$\hat D_{\alpha,\beta}\subset\Diff^1([0,1])$ of
diffeomorphisms $f$ whose fixed points are precisely $0$ and $1$ and
such that $Df(0)=\alpha$ and $Df(1)=\beta$ (this set contains
$D_{\alpha,\beta}$, and it differs from $D_{\alpha,\beta}$ because
we do not require $f$ to be affine in the neighborhood of $0$ and $1$).
Notice that $\hat D_{\alpha,\beta}$ is invariant by conjugacy by
orientation preserving diffeomorphisms of $[0,1]$.

\begin{conj}Every $C^1$-conjugacy class is dense in $\hat D_{\alpha,\beta}$. In other words, given any two elements $f,g\in \hat D_{\alpha,\beta}$,
there is a diffeomorphism arbitrarily $C^1$-close to $g$ that is conjugated to $f$ by a diffeomorphism of $[0,1]$.
\end{conj}
A positive answer to this conjecture would allow us to show that
every pathological behavior (in particular of the centralizer) built
on a Morse-Smale example would appear densely in $\Diff^1(S^1)$. It
would be interesting to understand the same question for
diffeomorphisms with irrational rotation number:

\begin{ques} Is every $C^1$-conjugacy class dense in the set of diffeomorphisms having a given rotation number $\alpha\in\RR\setminus \QQ$?
\end{ques}

\section{Local embeddability in a flow on $S^2$}\label{ss.sphere}
As in the one-dimensional case, the idea here is to measure how far
certain diffeomorphisms of $S^2$ are from the time one map of a
vector field. One obtains in this way a generalization of the Mather
invariant, which in this setting is a diffeomorphism of $\TT^2$.
Such an invariant has already been constructed
\footnote{In~\cite{AY}, the authors
  write that the Mather invariant for a diffeomorphism of $S^2$ is always
isotopic to the identity, but this is not correct (their Proposition 1
contains an error). For this reason, we choose here to
build in detail the construction of this invariant on the sphere.} in~\cite{AY} by V. Afraimovich and T. Young,
and we now have to show that by a $C^1$-small perturbation of the dynamics, this invariant
vanishes.

\subsection{Preparation of diffeomorphisms in $\cO$}
Let $S^2$ be the unit sphere in $\RR^3$ endowed with the coordinates
$(x,y,z)$. We denote by $N=(0,0,1)$ and $S=(0,0,-1)$ the north and
the south poles of $S^2$. Notice that the coordinates $x,y$ define
local coordinates of $S^2$ in local charts $U_N$ and $U_S$ in
neighborhoods of $N$ and $S$.

The following straightforward lemma asserts that one may assume that the fixed points of any diffeomorphism
$f$ in the open set $\cO$ are $N$ and $S$ and that the derivative at these points are conformal maps.

\begin{prop}\label{p.preparation}
Consider a diffeomorphism $f\in\cO$. Then there is a smooth
diffeomorphism $h\colon S^2\to S^2$ such that $h(N_f)=N$, $h(S_f)=S$
are the fixed points of $g=hfh^{-1}$; furthermore, the derivatives
$D_Ng$ and $D_Sg$ are conformal linear maps, i.e., each of them is a composition
of a  rotation with  a homothety of ratio $\alpha>1$ and $\beta<1$, respectively.

Finally, any $C^1$-neighborhood of $g$ contains a diffeomorphism
$\tilde g$ such that there are neighborhoods $V_N\subset U_N$ and
$V_S\subset U_S$ of $N$ and $S$, respectively, such that the
expression of $\tilde g$ in the coordinates $(x,y)$ is $\tilde
g(x,y)=D_Ng(x,y)$ for $(x,y)\in V_N$ and $\tilde g(x,y)=D_Sg(x,y)$
for $(x,y)\in V_S$.
\end{prop}

\subsection{Space of orbits of a conformal linear map}\label{ss.spaceorbits}

Let $A\in GL(\RR,2)$ be a conformal matrix of norm $\alpha\neq 1$.
There exists $a\in[0,2\pi)$ such that $A=R_a\circ h_\alpha$
where $R_a$ is the rotation of angle $a$ and $b$ and $h_\alpha$
is the homothety of ratio $\alpha$.
Notice that, for all $n\in\ZZ$, the linear map $A$ is
the time one map of the vector field
$$X_{A,n}=\ln\alpha.(x\frac\partial{\partial
  x}+y\frac\partial{\partial y})+ (a+2\pi n).(x\frac\partial{\partial
  y}-y\frac\partial{\partial x}).$$

The orbit space $T_A = \RR^2\setminus\{0\}/A$ (of the
action of $A$ on $\RR^2\setminus\{0\}$) is a torus (diffeomorphic to
$T^2=\RR^2/\ZZ^2$); we denote by $\pi_A$ the canonical projection from
$\RR^2\setminus\{0\}$ onto $T_A$. Moreover, the vector fields
$$Z= 2\pi(x\frac\partial{\partial y}-y\frac\partial{\partial x})$$ and
$X_{A,n}$ project on $T_A$ to pairwise transverse commuting vector
fields, which we also denote by $Z$ and $X_{A,n}$; the orbits of both
flows are periodic of period $1$. Hence, for any pair $(Z,X_{A,n})$
there is  a diffeomorphism $\cL_{A,n}\colon T_A\to T^2=\RR^2/\ZZ^2$
sending $Z$ to
$\partial/\partial x$ and $X_{A,n}$ to $\partial/\partial y$; this
diffeomorphism is unique up to composition with a translation of $T^2$.
Furthermore, the diffeomorphisms $\cL_{A,m}\circ \cL_{A,n}^{-1}$ are affine maps of the torus $T^2$, for all $n,m\in\ZZ$, so that $T_A$
is endowed with a canonical affine structure (indeed the affine map $\cL_{A,m}\circ \cL_{A,n}^{-1}$ on $T^2$ is the map induced
by the matrix $\left(\begin{array}{cc} 1&n-m\\0&1\end{array}\right)$
composed with a translation).

Note that the orbits of $Z$ correspond to the positive generator of
 the fundamental group of $\RR^2\setminus \{0\}$; we denote by
 $\sigma$ the corresponding element of $\pi_1(T_A)$. Given any closed
 loop $\gamma\colon[0,1]\to T_A$, and any point
 $x\in\RR^2\setminus\{0\}$ with $\pi_A(x)=\gamma(0)$,
 there is a lift of $\gamma$ to a path in $\RR^2\setminus\{0\}$
joining $x$ to $A^k(x)$,
where $k$ is the algebraic intersection number of $\sigma$ with $\gamma$.
Finally, observe that the homotopy classes corresponding to the orbits of $X_{n,A}$, when $n\in\ZZ$ are precisely those
whose intersection number with
$\sigma$ is $1$: in other words, there is a  basis of $\pi_1(T_A)=\ZZ^2$ such that $\sigma=(1,0)$ and the orbits of $X_{A,n}$ are homotopic to $(n,1)$.

\subsection{A Mather invariant for diffeomorphisms of $S^2$}

Denote by $D_{A,B}\subset \cO$ the set of diffeomorphisms $f\in \cO$ whose expression in the coordinates $(x,y)$ coincides with some conformal linear
maps $A$ and $B$ in neighborhoods $U^N_f$ of $N$ and $U^S_f$ of $S$.
The aim of this part is to build a Mather invariant for diffeomorphisms
in $D_{A,B}$.

Consider $f\in D_{A,B}$.
We retain the notations of the previous subsection and introduce
the affine tori $T_A$ and $T_B$.
The orbit space $\left(S^2\setminus\{N,S\}\right)/ f$ is a torus $T_f$ and we denote
by $\pi_f\colon S^2\setminus\{N,S\}\to T_f$ the natural projection. Furthermore, as $f$ agrees with $A$ on $U^N_f$, the torus $T_f$
may be identified
with the torus $T_A$ by a diffeomorphism $\pi_N\colon T_f\to T_A$, and in the same way, the fact that $f$ coincides with $B$ in a neighborhood of $S$
induces a diffeomorphism $\pi_S\colon T_f\to T_B$.

Notice that the homomorphisms $\pi_{N*} \colon H_1(T_f,\ZZ) \to H_1(T_A,\ZZ)$
and $\pi_{S*} \colon$ $H_1(T_f) \to H_1(T_B)$ preserve the homology class of $\sigma$ (corresponding to the positive homology generator
of $S^2\setminus \{N,S\}$ or of $\RR^2\setminus \{0\}$), and the homology intersection form with $\sigma$.

Consequently, for any $f\in D_{A,B}$, there is an integer $n(f)$ such that
the map $\De_{f,0,0}=\cL_{B,0}\circ \pi_S\circ \pi_N^{-1} \circ\cL_{A,0}^{-1}$ is isotopic to the linear map of $T^2$ induced
by the matrix $\left(\begin{array}{cc} 1&n(f)\\0&1\end{array}\right)$.

\begin{lemm} For any $f\in D_{A,B}$ there is a $C^1$-neighborhood $\cU$ of $f$ in $\Diff^1(S^2)$ such that for any $g\in \cU\cap D_{A,B}$
one has $n(f)=n(g)$.
\end{lemm}
\begin{proof}
We can choose a neighborhood $\cU$ such that, if $g\in\cU$ then the map
$$f_t(x)= \frac{(1-t)f(x)+t g(x)}{\|(1-t)f(x)+t g(x)\|}$$
is a smooth isotopy between $f$ and $g$.
Furthermore, by shrinking $\cU$ if necessary, for any
$g\in\cU$, the isotopy $f_t$ belongs to $\cO$ (that is $\Om(g)=\{N_g,S_g\}$).

If $g\in \cU\cap D_{A,B}$ then there are discs $D^N$ and $D^S$ centered on $N$ and $S$, respectively, such that $f_t=A$ on $D^N$ and $f_t=B$ on $D^S$
so that $f_t\in D_{A,B}$. In particular $f_t(D^S)\subset D^S$, and $f_t^{-1}(D^N)\subset D^N$. Furthermore, there exists $\ell>0$ such that for any
$x\in S^2\setminus (D^N\cup D^S)$,  $f_t^\ell(x)\in D^S$ and $f_t^{-\ell}(x)\in D^N$.

Let $x\in D^N$ such that $A(x)=f_t(x)\in D^N$ and $A^2(x)\notin
 D^N$. Hence $y_t=f_t^{\ell+2}(x)\in D^S$ and $f_t(y_t)=B(y_t)\in
 D^S$. Let $\gamma$ be the segment of orbit of $X_{A,0}$  joining $x$
 to $A(x)=f_t(x)$, and let $\gamma_t= f_t^{\ell+2}(\gamma)$.
For every $t$, $\gamma_t$ is homotopic (relative to $\{y_t,B(y_t)\}$
in $S^2\setminus \{N,S\}$) to a segment of orbit of $X_{B,n(f_t)}$. As a consequence, $n(f_t)$ varies continuously with $t$ as $t$ varies from $0$ to $1$. Hence $n(f_t)$ is constant; that is, $n(g)=n(f)$.
\end{proof}

Hence there is a partition of $D_{A,B}$ into open subsets $D_{A,B,n}$
such that $n(f)=n$ for $f\in D_{A,B,n}$. For $f\in D_{A,B,n}$, we
define:
$$\De_f=\cL_{B,n}\circ \pi_S\circ \pi_N^{-1} \circ\cL_{A,0}^{-1}.$$
Then $\De_f$ is a diffeomorphism of $T^2$, isotopic to the identity.
\medskip

Theorem~\ref{t.mather2dim} below justifies calling $\De_f$ \emph{the Mather invariant of $f$}.

\begin{theo}\label{t.mather2dim} Let $f\in D_{A,B,n}$ be a smooth diffeomorphism such that $\De_f$ is a translation of the torus $T^2$. Then $f$ leaves invariant two transverse  commuting vector fields $Z_f$ and $X_f$ on $S^2$ such that $Z_f=Z$ in a neighborhood of $\{N,S\}$, $X_f=X_{A,0}$ in a neighborhood of $N$ and $X_f=X_{B,n}$ in a neighborhood of $S$.

As a consequence the centralizer of $f$ is isomorphic to $S^1\times \RR$.
\end{theo}
\begin{proof}
Fix two discs  $D^N$ and $D^S$ centered at  $N$ and $S$,
respectively,  in which $f$ coincides with $A$ and $B$, respectively.

For any $x\neq S$ there exists $m(x)<0$ such that $f^{m(x)}(x)\in
D^N$. One defines $Z_f(x)=f^{-m(x)}_*(Z(f^{m(x)}))$ and
$X_f(x)=f^{-m(x)}_*(X_{A,0}(f^{m(x)}))$. As $Z$ and $X_{A,0}$ are
invariant by $A$, one proves that the vectors $Z_f(x)$ and $X_f(x)$
are independent of the choice of $m(x)$. As a consequence, one deduces
that they depend smoothly on $x\in S^2\setminus \{S\}$ and that they
commute on $S^2\setminus\{S\}$. Furthermore the restrictions of $Z_f$
and $X_f$ to $D^S$ are invariant by $f$, and hence by $B$, so that
they induce two vector fields on $T_B$ whose images by $\cL_{B,n}$ are
$\De_f(\frac{\partial}{\partial x})= \frac{\partial}{\partial x}$ and
$\De_f(\frac{\partial}{\partial y})=\frac{\partial}{\partial y}$,
respectively; that is, they
agree with the projections of the restrictions $Z$ and $X_{B,n}$ to
$D^S$.
Thus $Z_f=Z$ and $X_f=X_{B,n}$ on $D^S$, proving that the centralizer $C(f)$ contains
the subgroup of linear conformal maps $\{Z_f^t\circ X_f^t,\; (t,s)\in S^1\times \RR\}$.

Conversely, the following easy lemma shows that any diffeomorphism $g\in C(f)$ agrees with
a linear conformal map $h$ in a neighborhood of $S$. Since $g$ and $h$ both commute with $f$, one deduces
that $g$ and $h$ coincide on the whole sphere. Hence $C(f)$ is the group of conformal linear maps,
which is isomorphic to $S^1\times \RR$.
\end{proof}

\begin{lemma}\label{l.local}
Let $B=R_b\circ h_\beta$, with $b\in [0,2\pi)$ and $\beta\neq 1$ be a linear conformal contraction
whose angle $b$ is different from $0$ and $\pi$. Then, any local diffeomorphism $g$
defined in a neighborhood of $0$ and that commutes with $B$ coincides with a conformal linear map.
\end{lemma}

\subsection{Vanishing of the Mather invariant}\label{ss.vanish2}
This part is now very close to the $1$-dimensional case: we consider
$f\in D_{A,B,n}$, a disk $D^S_f$ centered on $S$ on which $f=B$
and want to perturb the homeomorphism $\Delta_f$.

Let $h\colon S^2\to S^2 $ be a diffeomorphism whose support is contained in a disk $D\subset D^S_f$, disjoint from all $B^m(D)$ for $m>0$.
The disk $D$ projects homeomorphically onto a disk  $D'\subset T_B$, and finally onto a disk $\tilde D=\cL_{B,n}(D')\subset T^2$. Let $\psi$ be the
diffeomorphism of $T^2$ with support in $\tilde D$ whose restriction to $\tilde D$ is the projection of $h$.  We says that $\psi$ is the projection of
$h$ on $T^2$ and conversely, that $h$ is the lift of $\psi$ with support in $D$.

Fix $k>0$ such that $D$ is disjoint from $B^k(D^S_f)$.

\begin{lemm}\label{l.compo} Consider a disk $D$ and a diffeomorphism $h$ as above.
Fix $k>0$ such that $D$ is disjoint from $B^k(D^S_f)$.
Then,  $f\circ h$ is a diffeomorphism in $D_{A,B,n}$
with $B^k(D^S_f)\subset D^S_{f\circ h}$, and whose
Mather invariant is $$\De_{f\circ h}= \psi\circ \De_f.$$
\end{lemm}
\begin{coro}\label{c.compo} Let $D_0,\dots, D_\ell\subset D^S_f$ be a finite sequence of disks such that
\begin{itemize}
\item for every $i$, the disk $D_i$ is disjoint from $B^k(D_i)$ for $k>0$;
\item for all $i<j$ the disk $D_i$ is disjoint from $B^k(D_j)$, $k\geq 0$
\end{itemize}
For every $i$, let $h_i$ be a diffeomorphism of $S^2$ with support in
$D_i$, and let $\psi_i$ be the projection of $h_i$ on $T^2$ (by $\cL_{B,n}\circ\pi_S\circ\pi_f$).

Then the Mather invariant of $f\circ h_0\circ\dots\circ h_\ell$ is
$$\De_{f\circ h_0\circ\dots\circ h_\ell}= \psi_\ell\circ\dots\circ \psi_0\circ \De_f.$$
\end{coro}
\medskip

Reciprocally, for any disk $\tilde D\subset T^2$ with diameter strictly
less than $1$, each connected component of
$(\cL_{B,n}\circ\pi_S\circ\pi_f)^{-1}(\tilde D)$ projects
diffeomorphically onto $\tilde D$, and $f$ induces a permutation of
these components. For $i>0$, let $D_i$ denote the (unique) component
of $(\cL_{B,n}\circ\pi_S\circ\pi_f)^{-1}(\tilde D)$ such that $
f^{-i}(D_i)\subset D^S_f$ but $f^{-(i+1)}(D_i)$ is not contained in $D^S_f$.
For any diffeomorphism $\psi$ with support in $\tilde D$  we will denote
by  $\theta_i(\psi)\colon S^2\to S^2$ the lift of $\psi$ with support in $D_i$.
\medskip

The next lemma is the unique reason we required that the derivative of $f$ at $N,S$ be complex, hence conjugate to conformal linear maps:

\begin{lemma} Let $\tilde D\subset T^2$ be a disk with diameter
  strictly less than $1$ and let $i,j\in\NN$. Then:
$$ \sup_{x\in S^2 }\|D_x\theta_i(\psi)-Id\| =\sup_{x\in S^2}\|D_x\theta_j(\psi)-Id\|.$$
\end{lemma}
\begin{proof} $\theta_i(\psi)$ is conjugated to $\theta_j(\psi)$ by $B^{j-i}$, which is the composition of a homothety
with a rotation; the $C^1$-norm is preserved by conjugacy by
isometries, and by homotheties, and hence is preserved by the
conjugacy by $B^{j-i}$.
\end{proof}

\begin{coro}\label{c.varep} For any $\varepsilon>0$ there is a $C^1$-neighborhood $\cV_\varepsilon\subset \Diff(T^2)$ of the identity
map such that for any diffeomorphism $\psi\in\cV_\varepsilon$ with support in a disk $\tilde D\subset T^2$ with diameter strictly less than
$1$, and for any $i\geq 0$, the lift $\theta_i(\psi)$ satisfies :
$$ \sup_{x\in S^2}\|D_x\theta_i(\psi)-Id\| <\varepsilon.$$
\end{coro}

\begin{defi} Let $\psi_1,\dots,\psi_\ell$ be $\ell$ diffeomorphisms of
$T^2$ such that the support of every $\psi_i$ is contained in a disk
$\tilde D_i$ with diameter strictly less than $1$;  a {\em lift of the sequence $\psi_1,\dots,\psi_\ell$} is a sequence of lifts
$h_1= \theta_{i_1}(\psi_1),\dots,h_\ell= \theta_{i_\ell}(\psi_\ell)$ such that, for every $i<j$ the support of $h_i$ is disjoint
from all the iterates $B^k(supp(h_j))$, for $k\geq 0$.
\end{defi}

It is straightforward to check that, for any sequence $\psi_1,\dots,\psi_\ell$ of diffeomorphisms of $T^2$ such that the support of every $\psi_i$ is
contained in a disk $\tilde D_i$ with diameter strictly less than $1$, the sequence $h_i=\theta_{i}(\psi_i)$ is a lift.
\bigskip

\begin{proof}[\noindent Proof of Theorem~\ref{t=sphere}]
Consider  $f\in D_{A,B,n}$ and a $C^1$-neighborhood $\cU$ of $f$. Fix
$\varepsilon>0$ such that, if $g_1,\dots,g_m$, $m>0$, are
diffeomorphisms of $S^2$ with pairwise disjoint supports in
$S^2\setminus \{N,S\}$, and such that
$ \sup_{x\in S^2} \|Dg_i(x)-Id\| <\varepsilon$, then $f\circ
g_1\circ\cdots\circ g_m\in\cU$. Let $\cV_\varepsilon$ be the
$C^1$-neighborhood
of the identity map of $T^2$ given by Corollary~\ref{c.varep}.

Using Theorem~\ref{t.decomposition}, we write
$$\De_f=\psi_1^{-1}\circ\cdots\circ \psi_\ell^{-1},$$
for some $\ell>0$, where $\psi_i\in \cV_\varepsilon$, and
the support of $\psi_i$ is contained in a disk $\tilde D_i$
with diameter strictly less
than $1$.
Let $(h_1,\dots,h_\ell)$ be a lift of the sequence
$(\psi_1,\dots,\psi_\ell)$; the $h_i$ satisfy
$$ \sup_{x\in S^2} \|D_xh_i-Id\| <\varepsilon,$$
by our choice of $\cV_\varepsilon$.

Our choice of $\varepsilon>0$ implies that
$g=f\circ h_1\circ\cdots\circ h_\ell$ is a diffeomorphism belonging
to $D_{A,B,n}\cap\cU$. Furthermore, its Mather invariant is
$\De_g=\psi_\ell\circ\cdots\circ\psi_1\circ \De_f=Id$.

We have just shown that any $f\in D_{A,B}$ is the $C^1$-limit of a
sequence $g_k\in D_{A,B}$ whose Mather invariant is the identity map;
in particular, the centralizer of $g_k$ is isomorphic to $\RR\times S^1$.

Since by Proposition~\ref{p.preparation}, $\cO$ contains a dense set
of diffeomorphisms smoothly conjugate to elements of $D_{A,B,n}$,
any diffeomorphism in $\cO$ is the limit of diffeomorphisms $g_k$
that are the time $1$ map of Morse-Smale vector fields, ending the
proof of Theorem~\ref{t=sphere}.
\end{proof}

\section{Local embeddability in a flow on surfaces}\label{s.surface}

The aim of this section is to prove Theorem~\ref{t.surface}:
we consider a closed connected surface $S$, and let $\cO\subset \Diff^1(S)$ be the set of Morse-Smale diffeomorphisms
as defined in the statement of Theorem~\ref{t.surface}.
Since Morse-Smale systems are structurally stable,
we have that  $\cO$ is a non-empty open subset of $\Diff^1(S)$.

Let $\cD_1\subset \cO$ be the dense subset of $\cO$ such that for every $f\in\cD_1$ one has:
\begin{itemize}
\item every fixed point $x$ of $f$ admits a neighborhood $U_x$ and smooth local coordinates defined
on $U_x$ such that the expression of $f$ in restriction to $U_x\cap f^{-1}(U_x)$ is linear
(hence coincides with the derivative $Df(x)$);
\item given any two fixed points $x,y$ of $f$, one has $\det Df(x)\neq \det Df(y)$.
\end{itemize}

The eigenvalues of every sink or source $q$ of a diffeomorphism
$f\in\cD_1$  are non-real. Hence we can choose the local coordinates
on $U_q$ in such a way that the restriction of $f$ to $U_q\cap f^{-1}(U_q)$
is a conformal linear map (i.e. the composition of a rotation with a homothety).
In the same way we can choose local
coordinates around any saddle $p$ so that $f\mid_{U_p\cap f^{-1}(U_p)}$
is described by a diagonal matrix.

\begin{prop}\label{p.surface} Given any $f\in \cD_1$ and any $C^1$-neighborhood $\cU$ of $f$, there is $g\in\cU$ such that $g$
is the time one map of a smooth vector field and agrees with $f$
outside an arbitrarily small neighborhood of
the sinks and sources of $f$.\\
Furthermore, the centralizer $C(g)$ of $g$ is isomorphic to $\RR$.
\end{prop}
Proposition~\ref{p.surface} clearly implies Theorem~\ref{t.surface}.
The proof of this proposition is the aim of the rest of
Section~\ref{s.surface}. In particular, up to Subsection~\ref{ss.endcontained}, we prove that $C(g)$
contains $\RR $, while in Subsection~\ref{ss.equality} we prove the equality.

We fix now a diffeomorphism $f\in \cD_1$ and and a neighborhood
$\cU$ of $f$. Let $\si_f$, $\alpha_f$ and $\omega_f$  denote the
sets of saddles, sources and sinks  of $f$, respectively.

\subsection{Vector field in a neighborhood of any saddle}

Recall that by assumption $\si_f$ is non-empty.
In this section, we shall build an invariant neighborhood of~$\si_f$ endowed with a flow
which will be our local model around saddles
(and their invariant manifolds).

\begin{prop}\label{p.saddle}
There exists an invariant open neighborhood $V_0$ and a vector field $Y_0$ on $V_0$ such that
\begin{itemize}
\item the flow of $Y_0$ is complete (i.e. defined from $-\infty$ to $+\infty$);
\item the diffeomorphism $f$ coincides with the time one map of $Y_0$ on $V_0$.
\end{itemize}
\end{prop}
\begin{proof}
We will use the following property satisfied by Morse-Smale diffeomorphisms:
\begin{description}
\item[(*)] \quad {\em For any two saddles $p_1,p_2\in \sigma_f$, there exist neighborhoods $B_{p_1},B_{p_2}$ of $p_1$ and $p_2$, respectively,
such that there is no point $x\in S\setminus (U_{p_1}\cup U_{p_2})$ whose backward orbit intersects $B_{p_1}$
and whose forward orbit intersects $B_{p_2}$.}
\end{description}
Let $p\in \sigma_f$ be any saddle periodic point.
One may assume that in the local coordinates $(x,y)$ of $U_p$, the expression of the map $f$
is $(x,y)\mapsto (\lambda^u_p x, \lambda^s_py)$.
Since the eigenvalues $\lambda^u_p,\lambda^s_p$ are positive,
$f$ agrees in $U_p$ with the time one map of the vector field
$$
Y_p(x,y) = x \ln \lambda_p^u \frac{\partial}{\partial x} + y \ln \lambda_p^s \frac{\partial}{\partial y}.
$$
For $T>0$ large, we introduce the octagon $\Delta_p$ defined by the equations:
$$|x|<(\lambda^u_p)^{-T},\;|y|<(\lambda_p^s)^T,\;
\frac{\ln|x|^{-1}}{\ln \lambda^u_p}+\frac{\ln |y|^{-1}}{\ln (\lambda^s_p)^{-1}}>3T.$$
By property (*), the forward orbit of $f(\Delta_p)\setminus \Delta_p$ does not intersect $\Delta_p$.
Hence, one can extend the vector field $Y_p$ to $V_p=\cup_{k\in \ZZ}f^k(\Delta_p)$, so that it is
equal to $f_*^k(Y_p)$ on $f^k(\Delta_p)\setminus f^{k-1}(\Delta_p)$
and to $f_*^{-k}(Y_p)$ on $f^{-k}(\Delta_p)\setminus f^{-k+1}(\Delta_p)$ for any $k>0$.
The open set $V_p$ is invariant by the flow of $Y_p$, and by construction the restriction of $f$ to $V_p$
coincides with the time one map of $Y_p$.

Using (*) again, we deduce that if each domain $\Delta_p$ has been chosen small enough, then the open sets
$V_p$ are pairwise disjoint. Hence, one can define on the union $V_0=\bigcup_{p\in\sigma_f}V_p$
a vector field $Y_0$ as required that coincides with $Y_p$ on any $V_p$.
\end{proof}

\subsection{Vector field on the punctured surface $S\setminus (\omega_f \cup \alpha_f)$}\label{ss.punctured}
In this section we prove the following proposition:

\begin{prop}\label{p.punctured} There exists a vector field $Y$ defined
on the punctured surface $S\setminus(\omega_f \cup \alpha_f)$ such that:
\begin{itemize}
\item the flow of $Y$ is complete;
\item the diffeomorphism $f$ is the time one map of $Y$ on $S\setminus(\omega_f \cup \alpha_f)$.
\end{itemize}
\end{prop}

Let $q\in\omega_f$ be a sink of $f$. In the local coordinates we
fixed on $U_q$, the expression for the diffeomorphism $f$ is given
by a conformal matrix $B=R_b\circ h_\beta$ with $b\in [0,2\pi)$ and
$0<\beta<1$. Hence the results of Subsection~\ref{ss.spaceorbits}
apply: $f$ is in the neighborhood of $q$ the time one map of each of
the vector fields $X_{B,n}$ for any $n\in \ZZ$. We also defined the
vector field $Z= 2\pi(x\frac\partial{\partial
y}-y\frac\partial{\partial x})$.

Denote by $\pi_q$ the projection of $W^s(q)\setminus \{q\}$ on the orbit space
$T_q= (W^s(q)\setminus\{q\})/f$, which is a torus $T^2$.
The vector fields $X_{B,n}$ and $Z$ project to vector fields whose orbits are all periodic.
We can choose a basis for the homology $H_1(T_q,\ZZ)$ such that the class of the orbits of $Z$ is $\sigma=(1,0)$ and
the class of the orbits of $X_{B,n}$ is $(n,1)$.

Now consider an invariant open neighborhood $V_0$ and a vector field $Y_0$ on $V_0$
as given by Proposition~\ref{p.saddle}.
We can also choose a smaller neighborhood $V$ of $\sigma_f$ that is invariant by $Y_0$
(by taking small balls centered at each saddle and saturating by the flow of $Y_0$).
We want to focus on the traces of $V$ and $V_0$ on the orbit spaces $T_q$.
We emphasize the following facts.
\begin{itemize}
\item[-] The set $\pi_q(V_0\cap W^s(q)\setminus\{q\})$ is foliated by the orbits of the projection $(\pi_q)_*Y_0$,
which are closed, have period $1$ and define the same non-zero homology class.
\item[-] As $f$ is not a north-south diffeomorphisms on the sphere,
there is at least one unstable separatrix of a saddle $p$ that is contained in the basin of~$q$.
\item[-] The set $\pi_q(V\cap W^s(q)\setminus\{q\})$ is a neighborhood of the projection in $T_q$
of the unstable separatrices that are contained in the basin of~$q$.
Hence it is compactly contained in $\pi_q(V_0\cap W^s(q)\setminus\{q\})$ and invariant by the flow of $(\pi_q)_*Y_0$.
\end{itemize}

This implies the following:

\begin{lemm}\label{l.vectorus}
The orbit space $T_q$ can be endowed with a vector field $Q_q$ such that
\begin{itemize}
\item the restriction of $Q_q$ to $\pi_q(V\cap W^s(q)\setminus\{q\})$ coincides with~$(\pi_q)_*(Y_0)$;
\item all the orbits of $Q_q$ are closed and of period~$1$.
\end{itemize}
\end{lemm}

Each vector field $Q_q$ lifts to a vector field $Y_q$ on the open
set $W^s(q)\setminus \{q\}=\pi^{-1}_q(T_q)$, and by construction
coincides with $Y_0$ on the intersection of $W^s(q)\setminus \{q\}$
with $V$. Hence, we have defined a vector field $Y$ on $V\cup
\bigcup_{q\in \omega_f}W^s(q)$ whose time one map agrees with $f$.
Any point $y$ of $S\setminus(\omega_f \cup \alpha_f \cup V)$ is a
wandering point and its $\omega$-limit set is a sink (otherwise $y$
would belong to one of the invariant manifolds of a saddle, hence to
$V$, which is a contradiction). This shows that $Y$ is now defined
on the whole punctured surface $S\setminus(\omega_f \cup \alpha_f)$,
proving Proposition~\ref{p.surface}.

\subsection{Gluing the vector fields around sinks and sources}\label{ss.endcontained}
The aim of this section is to perform a small
perturbation of $f$ in a small neighborhood of the sinks and sources
of $f$ (but keeping $f$ unchanged in a smaller neighborhood of the
sinks and sources) so that the vector field $Y$ provided by Proposition~\ref{p.punctured} can be extended
to a smooth vector field on $S$.

Fix a sink $q\in \omega_f$ and keep the notations of the previous section.
The dynamics in a neighborhood of $q$ agree with those of a conformal linear map $B$.
By projecting the vector field $Y$ on the torus $T_q$,
we obtain a vector field $\hat Y_q=(\pi_q)_*(Y)$.
Each orbit of $Y$ is a path joining a point $y$ to $f(y)$.
consequently, the orbits of $\hat Y_q$ on $T_q$ are (simple) curves
and are in the same homology class as the orbits of a vector field
$\hat X_q$ obtained by projecting the vector field $X_{B,n_q}$, for some $n_q$.

\begin{lemm}\label{l.isotop}
There exists a smooth diffeomorphism $\psi_{f,q}$ of $T_q$
that is isotopic to the identity map and such that $(\psi_{f,q})_*(\hat Y_q)=\hat X_q.$
\end{lemm}
\begin{proof}
The orbits of $\hat Y_q$ and $\hat X_q$ are all periodic of period $1$ and are in the same homology class.
Let $\sigma$ be a cross-section of $\hat X_p$ obtained by projecting an orbit of the vector field
$Z= 2\pi(x\frac\partial{\partial y}-y\frac\partial{\partial x})$.
One chooses a complete smooth cross section $\sigma_q$
of $\hat Y_q$ that is in the same homology class as $\sigma$
and cuts every orbit of $\hat Y_q$ in exactly one point.
We also choose an orientation-preserving diffeomorphism $\psi_{f,q}\colon \sigma_q\to\sigma$.
This diffeomorphism extends in a unique way in the announced diffeomorphism of the torus $T_q$.
\end{proof}

\begin{rema}\begin{itemize}
\item The diffeomorphism $\psi_{f,q}$ is not unique: its depends on the choice of $Y$
and of the cross section $\sigma_q$ in the neighborhood of $q$.
\item If $\psi_{f,q}$ is the identity map, then the vector field $Y$ agrees with the linear vector fields $X_q$
in a neighborhood of $q$.
\end{itemize}
\end{rema}

Given any $f\in \cD_1$, we will perform a perturbation $g_q\in\cD_1$ of $f$
whose associated diffeomorphism $\psi_{g,q}$ is the identity map.

\begin{prop}\label{p.perturb} Consider $f\in\cD_1$, a smooth complete vector field $Y$
on $S\setminus (\alpha_f\cup\omega_f)$ such that $f$ is the time one map of the flow of $Y$
and a sink $q\in \omega_f$. For every $C^1$-neighborhood $\cU$ of $f$,
and for every neighborhood $O$ of $q$ there is $g_q\in \cU\cap\cD_1$ with the following properties:
\begin{itemize}
\item there is neighborhood $O'\subset O$ of $q$ such that $g_q=f$ on $O'\cup(S\setminus O)$;
\item there is a smooth complete vector field $\tilde Y$ defined on
$S\setminus(\alpha_f\cup\omega_f)\cup\{q\}$ coinciding with $Y$ outside $O$
and such that $g_q$ is the time one map of the flow of $\tilde Y$.
\end{itemize}
\end{prop}

One obtains the first part of Proposition~\ref{p.surface} by applying
Proposition~\ref{p.perturb} to $f$ for each sink  and to $f^{-1}$ for each source, successively.

The argument for modifying $\psi_{f,q}$
is almost identical to the proof of Theorem~\ref{t=sphere} at Section~\ref{ss.vanish2}, and we just sketch it.
\bigskip

\noindent
\begin{proof}[Sketch of the proof of Proposition~\ref{p.perturb}]
By reducing $O$, we may assume that it is a small disk centered at $q$ and contained in $U_q$.
Let $D\subset O$ be a disk whose iterates
$f^i(D)$, $i\in\NN$, are pairwise disjoint and all contained in
$O$. We denote by $\hat D=\pi_q(D)$ the projection of $D$ on
$T_q$. Notice that $\pi_q$ induces a diffeomorphism from $D$ to
$\hat D$. Let $h$ be a diffeomorphism of $S$ coinciding with the
identity map on the complement of $D$. We denote by $\hat h$ the diffeomorphism of
$T_q$ that is the identity map on the complement of $\hat D$ and is $\pi_q\times
h\times \pi_q^{-1}$ on $\hat D$.  Let $g_h= f\circ h$, and
let $Y_h$ be the vector field on $S\setminus (\alpha_f\cup \omega_f)$ that
coincides with $Y$ in the complement of $\bigcup_{i>0}f^i(D)$ and with $(f^i\circ
h)_*(Y)$ on $f^i(Q)$ for $i>0$.
The following lemma is  the analogue of Lemma~\ref{l.compo}.

\begin{lemm}\label{l.perturb} With the notation above, one has:
\begin{itemize}
\item $g_h$ belongs to $\cD_1$ and agrees with $f$ in a neighborhood of $Fix(f)$;
\item $g_h$ is the time one map of the flow of the vector field $Y_h$;
\item $(\psi_{f,q}\circ \hat h^{-1})_*(\hat Y_h)=\hat X_q$, so that one can choose $\psi_{g_h,q}=\psi_{f,q}\circ \hat h^{-1}$.
\end{itemize}
\end{lemm}
\begin{proof}
The unique difficulty here is to show that $g_h$ is the time one map
of $Y_h$. For that, let $H$ be the diffeomorphism of $S\setminus
(\alpha_f\cup\omega_f)$ that is $f^i\circ h\circ f^{-i}$ on $f^i(D$)
for $i>0$ and the identity map in the complement of
$\bigcup_{i>0}f^i(D)$. Note that $g_h$ is conjugate to $f$ by $H$
and $Y_h=H_*(Y)$.
\end{proof}

Conversely, for any disk $\hat D\subset T_q$ with small diameter, the connected
components of $\pi_q^{-1}(\hat D)$ are diffeomorphic to $\hat D$; we denote by $D_i$ the component
that is contained in $B^i(O)$ but not in $B^{i+1}(O)$.
For any diffeomorphism $\hat h$ of $T_q$, with support in the disk $\hat D$, we denote by $\cL_i{h}$
the lift of $\hat h$ that is supported in $D_i$.

Let now consider a diffeomorphism $\psi_{f,q}$ of $T_q$ associated
to $f$ and $Y$. Since $\psi_{f,q}$ is isotopic to the identity map,
the fragmentation lemma (Theorem~\ref{t.decomposition}) allows us to
write $\psi_{f,q}$ as the composition $\psi_{f,q}=\hat
h_k\circ\cdots\circ \hat h_1$ of finitely many diffeomorphisms $\hat
h_i$ arbitrarily $C^1$-close to the identity map and each with
support in an arbitrarily small disk. We then build the lifts $h_i=
\cL_{2i}(\hat h_i)$ whose supports are pairwise disjoint. Let
$g=f\circ h_1\circ\cdots\circ h_k$. Then $g$ belongs to $\cD_1$ and
agrees with $f$ in the complement of $O$ and on $f^{2k+3}(O)$.
Furthermore, since the diffeomorphisms $\hat h_i$ can be chosen
close to the identity, $g$ belongs to $\cU$. Finally, applying
inductively Lemma~\ref{l.perturb}, we see that $g$ is the time one
map of a vector field $\tilde Y$ that coincides with the linear
vector field $X_q$ in a neighborhood of $q$. Hence it may be
extended smoothly at $q$, ending the proof.
\end{proof}

\subsection{End of the proof of theorem~\ref{t.surface}}\label{ss.equality}
Let $\cD_2\subset \cD_1$ be the subset of diffeomorphisms that
are the time one map of the flow of a smooth vector field. The
previous sections proved that $\cD_2$ is dense in $\cD_1$, hence in
the open subset  $\cO$ of $\Diff^1(M)$. In order to prove
Theorem~\ref{t.surface}, we will first compute $C(f)$ for
$f\in\cD_2$.

Let $f\in \cD_2$ be the time one map of a smooth vector field $X$,
and consider $g\in C(f)$. By our assumptions on $\cD_1$
(two different fixed points have different determinant), $f$ and $g$ have the same fixed points.
Consider now any saddle $p$ of $f$. The \emph{unstable separatrices}
are defined to be the connected components of $W^u(x)\setminus \{x\}$, and we denote them by $W^s_+(x)$, $W^s_-(x)$.
Note that $g$ preserves or exchanges the two separatrices of $p$.

\begin{prop}\label{p.separatrice}
Consider $f\in \cD_2$.
If $g\in C(f)$ preserves an unstable separatrix of a saddle $p\in \sigma_f$,
then $g$ belongs to the flow of $X$.
\end{prop}

As in the proof of Proposition~\ref{p.centre-cercle} of Section~\ref{ss.centre-cercle}, we deduce:
\begin{coro} For $f\in\cD_2$ the centralizer is either the flow of $X$
(hence is isomorphic to $\RR$) or is isomorphic to $\RR\times\ZZ/2\ZZ$.
\end{coro}

Every unstable separatrix of a saddle $p$ is contained in the basin of a sink $q$ of $f$.
In this section, we endow each orbit space $T_q$ with affine coordinates $(r,s)\in \RR^2/\ZZ^2$
such that the vector field $X$ projects to $\frac\partial{\partial s}$,
and the vector field $Z$, whose expression in the local coordinates at $q$
is $Z= 2\pi(x\frac\partial{\partial y}-y\frac\partial{\partial x})$,
projects to $\frac \partial {\partial r}$.
Note that the unstable separatrix of $p$ is precisely one orbit of $X$
and induces on $T_{q}$ a circle $\{r\}\times \RR/\ZZ$.

\begin{proof}[Proof of proposition~\ref{p.separatrice}]
Assume that $g(W^u_+(p))=W^u_+(p)$, for some saddle point $p$ and
let $q$ be the sink of $f$ whose basin contains this separatrix.
The projection of the separatrix $W^s_{+}(p)$ on $T_{q}$
will be denoted by $\{r_0\}\times \RR/\ZZ$.

Since $g$ commutes with $f$, it induces on $T_{q}$ a
diffeomorphism $g/_f$. By Lemma~\ref{l.local}, $g$ is locally the composition of an
homothety and a rotation; hence the expression of $g/_f$ in the $(r,s)$
coordinates is a translation: $(r,s)\mapsto (r+\alpha,s+\beta)$.

The fact that $g$ leaves invariant $W^s_+(p)$ implies that
$g/_f( \{r_0\}\times \RR/\ZZ)=\{r_0+\alpha\}\times
\RR/\ZZ=\{r_0\}\times \RR/\ZZ$, so that $\alpha=0$. In
particular, this implies that $g/_f$ leaves invariant every orbit of
$\frac\partial{\partial s}$.  Consequently, $g$ leaves invariant
every $X$-orbit contained in the basin $W^s(q)$.

Let $q_1$ be a source such that $W^u(q_1)\cap W^s(q_0)\neq\emptyset$.
In the same way we endow the torus $T_{q_1}$ with affine coordinates,
and $g$ induces on $T_{q_1}$ a translation. The set $W^u(q_1)\cap
W^s(q_0)\neq\emptyset$ is open and invariant by $X$;
hence it contains an $X$-orbit, which is invariant by $g$. This orbit
induces on $T_{q_1}$ a circle of the form $\{r_1\}\times \RR/\ZZ$,
invariant by $g/_f$. This proves as before that $g$ leaves invariant
every $X$-orbit contained in $W^u(q_1)$.

Since $S$ is connected and $f$ is Morse-Smale,
for every sink or source $q$ there
is a finite sequence $q_0,q_1,\ldots,q_n=q$ of alternating sources and sinks
such that $W^u(q_i)\cap W^s(q_{i+1})$ or $W^s(q_i)\cap W^u(q_{i+1})$
is non empty for each $i= 0,\ldots,n-1$.
The discussion above hence proves that $g$ leaves invariant every
$X$-orbit contained in the basin or a sink or of a source, hence
leaves invariant every $X$-orbit.

This shows that for every point $x\in S\setminus Fix(f)$, there exists
$t(x)\in\RR$ such that $g(x)=X^{t(x)}(x)$. Futhermore, the continuous map
$x\mapsto t(x)$ is locally constant in the punctured neighborhoods
of the sinks and of the sources; hence (using the fact that $g$ and
$X$ commute with $f$) the map $t$ is constant on every basin of a sink or a source.
Since $t$ takes the same value on any two intersecting basins,
it follows that $t(x)$ is constant on the complement of the fixed points.
Thus $g$ belongs to the flow of $X$.
\end{proof}

\bigskip

In order to conclude the proof of Theorem~\ref{t.surface}, it remains to
show that there is a dense subset $\cD\subset \cD_2$ of
diffeomorphisms such that, for each $f\in \cD$, $C(f)$ is precisely the flow of the
corresponding vector field.
We first note that for the existence of an extra symmetry, it is necessary that for
any saddle $p\in \sigma_f$, the two unstable separatrices belong to the basin
of a same sink $q$. Assuming that this topological condition is satisfied,
the proof is very similar to the argument on the circle:
we will exhibit a new invariant of differentiable conjugacy
which vanishes if the diffeomorphism has an extra symmetry.
We will then show that arbitrarily small perturbations
allow us to modify this invariant.

We first define precisely this invariant.
For every sink $q\in\omega_f$, we consider the
orbit space $T_q$ of $W^s(q)\setminus\{q\}$, with its structure of an affine torus.
For every saddle $p$ whose unstable manifold has a
unique sink $q$ in the $\omega$-limit set, let $\{r_+(p)\}\times S^1$
and $\{r_-(p)\}\times S^1$ be the projections of the
separatrices $W^u_+(p)$ and $W^u_-(p)$ on the affine torus $T_{q}$.
We consider the distance $|r_+-r_-|\in[0,\frac12]$  between the two points
$r_+,r_-\in S^1=\RR/\ZZ$.

With these notations, for every saddle $p$ in $\sigma_f$ we define
$$
\delta^u(p)   =  \left\{
\begin{array}{l}
\frac12-|r_+(p)-r_-(p)| \mbox{ if $W^u_+(p),W^u_-(p)$ have the same $\omega$-limit set;}\\
~\\
  + \infty \mbox{ otherwise.} \\
\end{array}
\right. 
$$
As we'll explain in the proof of Lemma~\ref{l.opposite}, the
number $\delta^u(p)$ measures whether the projections of the two
separatrices on an affine torus can be exchanged by a rotation.

The invariant of the dynamics we will work with is defined by
$$\delta^u(f)=\sup_{p\in \sigma_f} \delta^u(p).$$

We conclude the proof of Theorem~\ref{t.surface} with the following two lemmas.

\begin{lemm}\label{l.opposite} Let $f$ be in $\cD_2$.
If its centralizer $C(f)$ is isomorphic to $\RR\times \ZZ/2\ZZ$, then $\delta^u(f)=0$
\end{lemm}
\begin{proof} Let $g$ be the element of order $2$;
it exchanges the unstable separatrices of every saddle $p$ of $f$.
In particular, $W^u_+(p)$ and $W^u_-(p)$ are contained in the basin of the same sink $q$.
Moreover, $g$ projects on $T_{q}$ to the translation by $(1/2,0)$.
This implies that we must have $\delta^u(p)=0$ and we are done.
\end{proof}

\begin{lemm} For $f$ in a dense subset $\cD\subset \cD_2$
we have $\delta^u(f)\neq 0$.
\end{lemm}
\begin{proof} Let $f\in\cD_2$ be such that $\delta^u(f)=0$ and let $X$
be the flow associated with~$f$. The unstable separatrices of any saddle $p$
are contained in the basin of a sink $q$.
By an arbitrarily small perturbation of $X$ with support in
the complement of a neighborhood of the fixed points, one can change
the projection in $T_q$ of any of the two separatrices.
Then the time one map of the perturbed flow is a diffeomorphism $g\in\cD_2$
arbitrarily close to $f$ and such that $\delta^u(g)\neq 0$.
\end{proof}

\section{Huge centralizers in dimension larger than~$3$}

\subsection{Reduction to the existence of periodic islands}
Theorem~\ref{t.wild} is a consequence of the next result:

\begin{theo}\label{t.BD} Let $M$ be a compact manifold of dimension $d\geq 3$. Then there is a non-empty open subset $\cO\subset \Diff^1(M)$ and a dense part $\cD_0\subset \cO$ such that any diffeomorphism $f\in\cD_0$ has a periodic point $x$ such that $Df^{n}(x)= Id\in GL(T_xM)$ where $n$ is the period of $x$.
\end{theo}
 Before explaining the proof of Theorem~\ref{t.BD} we explain here why it implies Theorem~\ref{t.wild}.

\begin{coro}\label{c.wild} Let $M$ be a compact manifold of dimension $d\geq 3$. Then there is a non-empty open subset $\cO\subset \Diff^1(M)$ and a dense part $\cD\subset \cO$ such that any diffeomorphism $f\in\cD$ has the following property:

there is an embedded ball $D_f \subset M $ of dimension $d$ and an integer $n>0$ such that $f^i(D_f)\cap D_f=\emptyset$ for $i\in\{1,\dots,n-1\}$
and the restriction of $f^n$ to $D_f$ is the identity map.
\end{coro}
\begin{proof} Given any $f_0$ in the set $\cD_0$ given by Theorem~\ref{t.BD}, given any $C^1$-neighborhood  $\cU$ of $f_0$ and given any neighborhood $V$ of the orbit $Orb(x,f_0)$, there exists $f\in U$ that agrees with $f$ in the complement of $V$, and such that $f^n$ is the identity map in a neighborhood of $x$.
\end{proof}

\noindent\begin{proof}[End of the proof of Theorem~\ref{t.wild}] Consider $f$ in the set $\cD$ constructed in Corollary~\ref{c.wild}, and let $D$ be a periodic ball of period $n$ such that $f^n$ coincides with the identity map on $D$ (hence on the $f$-orbit of $D$). Let $\varphi\colon\RR^d\to M$ be a smooth embedding such that $\varphi(\DD^d)=D$. To any diffeomorphism $h\in\Diff^1(\RR^d,\RR^d\setminus \DD^d)$ we associate $h_0\colon M\to M$, the diffeomorphism  equal to the identity map in the complement of $D$ and equal to $\varphi h\varphi^{-1}$ on $D$. For each $i\in\ZZ$ we set $h_i= f^i\circ h_0\circ f^{-i}$; notice that $h_i$ is a diffeomorphism of $M$ with support contained in $f^i(D)$, and $h_{i+n}=h_i$. We denote by $h_\varphi\colon M\to M$ the diffeomorphism of $M$ that coincides with $h_i$ on $f^i(D)$, for every $i\in\ZZ$, and with the identity map in the complement of $\bigcup_i f^i(D)$.

By construction, $h_\varphi$ commutes with $f$. Then $h\mapsto h_\varphi$ is an injective homomorphism from $\Diff^1(\RR^d,\RR^d\setminus \DD^d)$ to $C(f)$.
\end{proof}

\subsection{Existence of periodic orbits tangent to  the identity map}

The open set $\cO$ produced in Theorem~\ref{t.BD} and Corollary~\ref{c.wild} is analoguous to those  built  in \cite{BD}.  We start by recalling some notions.

Let $f$ be a diffeomorphism, and let $x\in Per_{hyp}(f)$ a hyperbolic periodic point. Given another hyperbolic periodic point $y$ of $f$, we say that $x$ and $y$ are \emph{homoclinically related} and we write $x\sim y$ if the stable and the unstable manifolds of the orbit of $x$ transversely intersect the unstable and the stable manifolds of the orbit of $y$, respectively.  Let  $\Si(x,f)=\{y\in Per_{hyp}(f), y\sim x\}$. The \emph{homoclinic class} $H(x,f)$ is the closure $H(x,f)=\overline{\Si(x,f)}$.

A point $x\in M$ is \emph{chain recurrent} if for every $\delta>0$ there exists a $\delta$-pseudo-orbit $x=x_0,x_1,\dots, x_k=x$. The \emph{chain recurrence class $C(x,f)$} of a chain recurrent point $x$ is the set of points $y$ such that, for every $\delta>0$, there is a $\delta$-pseudo orbit starting at $x$ and ending at $y$ and a $\delta$-pseudo orbit starting at $y$ and ending at $x$.

For any periodic point $y\in Per(f)$, let $\pi(y)$ be its period, and let $$J_f(y)=\frac1{\pi(y)}\log| Det Df^{\pi(y)}(y)|,$$
be the sum of the Lyapunov exponents of $y$.

Recall  that an $f$-invariant   set $\La$ admits a \emph{dominated splitting} if there is an $Df$-invariant decomposition $TM|_\La=E\oplus F$ of the tangent bundle $TM$ over $\La$ as a direct sum of two invariant subbundles $E$ and $F$ such that:
\begin{itemize}
\item the dimension $dim(E(x))$ is independent on $x\in\La$;
\item the vectors in $E$ are uniformly less expanded than the vectors in $F$; that is, there exists $N\in \NN$ such that for any $x\in \La$ and any non-zero vectors $u\in E(x)$ and $v\in F(x)$:
$$\frac{\|Df^N(u)\|}{\|u\|}<\frac12\frac{\|Df^N(v)\|}{\|v\|}.$$
\end{itemize}
The bundles $E$ and $F$ of a dominated splitting are always continuous and extend continuously to a dominated splitting over the closure of $\La$ (elementary properties of dominated splitting are described in \cite[Appendix B.1]{BDV}). As a direct consequence, \emph{if a set does not admit a dominated splitting, then the same holds for any dense subset of it.}

Now Theorem~\ref{t.BD} follows from
\begin{prop}\label{p.nodomination0} Let $M$ be a compact manifold of dimension $dim(M)\geq 3$. There is a non empty open subset $\cO\subset \Diff^1(M)$ and a continuous function $f\in \cO\mapsto x_f\in M$ such that, for every $f\in \cO$:
\begin{itemize}
\item $x_f$ is a hyperbolic periodic saddle point of $f$ with $J_f(x)>0$;
\item there exists $y_f\in\Si(x_f,f)$ such that $J_f(y_f)<0$;
\item the chain recurrent class $C(x_f,f)$ does not admit a dominated splitting.
\end{itemize}
\end{prop}

We now deduce Theorem~\ref{t.BD} from Proposition~\ref{p.nodomination0}

\noindent\begin{proof}[Proof of Theorem~\ref{t.BD}] We just repeat briefly here the proof given in \cite{BD}. Fix a diffeomorphism $f\in \cO$, and a neighborhood $\cU\subset\cO$ of $f$. We will prove that $\cU$ contains a diffeomorphism $g$ having a periodic orbit whose derivative at the period is the identity.

According to \cite{BC}, for $C^1$-generic diffeomorphisms, the chain recurrent class of every periodic orbit is equal to its homoclinic class. Hence there exists $f_0\in\cU$ such that $C(x_{f_0},f_0)=H(x_{f_0},f_0)$.

For every $\varepsilon>0$ we consider the set $$\Si_\varepsilon(f_0)=\{y\in\Si(x_{f_0},f_0), |J_{f_0}(y)|<\varepsilon\}.$$
Any two points in $\Si(x_{f_0},f_0)$ are homoclinicaly related. As a consequence, given a finite set $X\subset \Si(x_{f_0},f)$, there is a hyperbolic basic set of $f_0$ containing $X$.
From this fact and from the hypotheses $J_{f_0}(x_{f_0})>0$ and $J_{f_0}(y_{f_0})>0$, we deduce that, for every  $\varepsilon>0$, the set $\Si_\varepsilon(f_0)$ satifies the two following properties:
\begin{enumerate}
\item the set $\Si_\varepsilon(f_0)$ is dense in $\Si(x_{f_0},f_0)$, and hence in $H(x_{f_0},f_0)=C(x_{f_0},f_0)$; it follows that $\Si_\varepsilon(f_0)$ does not admit a dominated splitting;
\item the set $\Si_\varepsilon(f_0)$ admits \emph{transitions} as defined in \cite{BDP}. This is implied by the fact that, given any finite subset $X\subset\Si_\varepsilon(f_0)$, there is a hyperbolic basic set  $K_X$ containing $X$ whose periodic orbits are contained in $\Si_\varepsilon(f_0)$: $K_X\cap Per(f_0)\subset \Si_\varepsilon(f_0)$.
\end{enumerate}

Since $\Si_\varepsilon(f_0)$ admits transitions and does not admit a dominated splitting, \cite{BDP} implies that, for every $\delta>0$, there is a periodic point $x\in\Si_\varepsilon(f_0)$ and a $\delta$-small perturbation $g_0$ of $f_0$ agreeing with $f_0$ on the orbit of $x$ (and in the complement of an arbitrarily small neighborhood of the orbit of $x$) such that $Dg_0^{\pi(x)}(x)$ is an homothety, where $\pi(x)$ is the period of $x$. Notice that, for $\delta>0$ small enough, we have $|J_{g_0}(x)|<2\varepsilon$.

As $\varepsilon$ can be chosen arbitrarilly small, for $\varepsilon<<\delta$ there is a $\delta$-small perturbation $g$ of $g_0$ coinciding with $g_0$ and $f$ on the orbit of $x$ (and in the complement  of an arbitrarily small neighborhood of the orbit of $x$) such that $Dg^{\pi(x)}(x)$ is the identity map.
For $\delta$ small enough, $g$ belongs to $\cU$ concluding the proof.
\end{proof}

It remains to explain how one can build the open set $\cO$ announced in Proposition~\ref{p.nodomination0}. Our construction (as in \cite{BD}) is based on the coexistence in a single homoclinic class of periodic orbits having complex eigenvalues of any rank. Let us explain this notion.

Let $f$ be a diffeomorphism and $x$ a periodic point of $f$ of period $\pi(x)$. An \emph{eigenvalue}  of $x$ is a (real or complex) eignevalue of the derivative $Df^{\pi(x)}(x)$.  We say that $x$ has a complex eigenvalue of rank $(i,i+1)$, for $i\in\{1,\dots,dim(M)-1\}$, if there is a $Df^{\pi(x)}$-invariant splitting $T_xM=E\oplus F\oplus G$ such that:
\begin{itemize}
\item $dim(E)=i-1$, $dim(F)=2$, $dim(G)=dim(M)-i-1$;
\item the restriction of $Df^{\pi(x)}(x)$ to $F$ has a pair of complex (non-real) conjugated eigenvalues; let $\lambda$ be the modulus of these eigenvalues;
\item the eigenvalues of the restrictions of $Df^{\pi(x)}(x)$ to $E$  (resp. $G$) have modulus strictly less (resp. larger) than $\lambda$.
\end{itemize}

\begin{lemm} Let $K$ be a $f$ invariant set. Assume that $K$ contains a periodic point $x$ having a complex eigenvalue of rank $(i,i+1)$. Assume that $TM=E\oplus F$ is a dominated splitting on $K$. Then $dim E\neq i$.
\end{lemm}
\begin{proof}Just notice that $E(x)$ and $F(x)$ are $Df^{\pi(x)}(x)$-invariant, and that the eigenvalues corresponding to $E(x)$ are strictly less that those corresponding to $F(x)$.
\end{proof}
 \begin{coro}Let $K$ be an $f$-invariant set containing periodic points $x_i$ having a complex eigenvalue of rank $(i,i+1)$ for every $i\in\{1,\dots, dim M-1\}$.
Then $K$ does not admit a dominated splitting.
\end{coro}

Hence Proposition~\ref{p.nodomination0} is a direct consequence of:
\begin{prop}\label{p.nodomination}For any compact manifold $M$ with $dim M\geq 3$, there is a non empty open subset $\cO\subset \Diff^1(M)$  and a continous map $f\in \cO\mapsto x_f\in M$ such that, for every $f\in\cO$:
\begin{enumerate}
\item the point $x_f$ is a hyperbolic periodic point of stable index (dimension of the stable manifold) equal to $1$;
\item $J_f(x_f)>0$ and there is a hyperbolic periodic point $y_f$, homoclinically related to $x_f$, such that $J_f(y_f)<0$
\item for every $i\in\{2,\dots,dim M-1\}$, there is a hyperbolic periodic point $x_{i,f}$ homoclinically related to $x_f$, and having a complex eigenvalue of rank $(i,i+1)$.
\item the chain recurrence class $C(x_f,f)$ contains a hyperbolic periodic point $z_f$ of stable index equal to $2$, and having a complex eigenvalue of index $(1,2)$.
\end{enumerate}
\end{prop}
\begin{proof} The properties described in items (1)(2)(3) are open properties, and are easy to get by a local argument: we just describe how to obtain a finite number of periodic saddles of index $1$ that are homoclinicaly related.

The property of item (4) comes from the notion of blender defined in \cite{BD}. A blender $\La$ is:
\begin{itemize}
\item a (uniformly) hyperbolic compact set $\La$ that is also partially hyperbolic: there is an invariant dominated splitting $E^s\oplus E^u\oplus E^{uu}$ over $\La$ such that the dimension of $E^u$ is equal to $1$; here we assume that the dimension of the stable bundle $E^s$ is $1$, so that $dim (E^{uu})=dim M-2$. The partially hyperbolic structure extends to a neighborhood $U$ of $\La$; and
\item an open region $V\subset U $ (called \emph{the caracteristic region of the blender}), endowed with a cone field $C^{uu}$ around a the bundle $E^{uu}$;
\end{itemize}
with the following property. There exists a $C^1$-neighborhood $\cU$
of $f$ such that, for every $g\in \cU$ and any ball $D^{uu}\subset
U$ of dimension  $dim(D^{uu})=dim(E^{uu})$, tangent to the conefield
$C^{uu}$, and crossing the region $V$, then $D^{uu}$ meets the
stable manifold of the continuation $\La_g$ of $\La$ for $g$. (See
\cite[Section6.2]{BDV} for a more detailed discussion of the notion
of blender, and references.)

Let $f$ be a diffeomorphism having
\begin{itemize}
\item a blender $\La$ containing a periodic point $x$ of index $1$ and such that $J_f(x)>0$,
\item a point $z_f$ of index $2$ such that:
\begin{itemize}
\item  the unstable manifold $W^u(z)$ crosses the caracteristic region of the blender, remaining tangent to the strong unstable cone field $C^{uu}$,
\item the stable manifold  $W^s(z)$ intersects transversely  $W^u(x)$,
\item the stable eigenvalue of $z$ is not real (hence $z$ has a complex eigenvalue of rank $(1,2)$),
\end{itemize}
\item for any $i\in \{2,...,dim M-1\}$, a hyperbolic periodic point $x_i$ of index $1$ homoclinicaly related to $x$ and having a complex eigenvalue of rank $(i,i+1)$,
\item a hyperbolic periodic point $y$ homoclinicaly related to $x$ and such that $J_f(y)<0$.
\end{itemize}

All these properties are robust. Hence  there is a small neighborhood $\cO$ of $f$ such that the continuations of the periodic points $x,y,z, x_i$ and of the blender $\La$ are well-defined for every $g\in\cO$ and satisfy all the properties above. We conclude by noting that, for every $g\in \cO$, the point $z_g$ belongs to the chain recurrence class of $x_g$.
\end{proof}

\section{Huge centralizers of symplectomorphisms}

The aim of this section is to prove that the set of symplectomorphisms having a large centralizer
is dense in $\cO\subset \Symp^1_\omega(M)$,
the open subset consisting of symplectomorphisms having a robustly totally elliptic periodic point.
The argument is analogous to the idea of Theorem~\ref{t.wild}.

\begin{prop}\label{p.symp}There is a dense part $\cD\subset \cO$ such that, for $f\in \cD$
there is a ball $D\subset M$ and a integer $n>0$ such that  $D\cap f^i(D)=\emptyset$ for $i\in\{1,\dots,n-1\}$,
and the restriction of $f^n$ to $D$ is the identity map.
\end{prop}
This proposition is a consequence of the two following classical lemmas, that can be easily obtained
by considering generating functions.

\begin{lemm}[Symplectic Franks Lemma]\label{l.franks} Given a symplectomorphism $f$ and a neighborhood $\cU\subset \Symp^1_\omega(M)$ of $f$, there exists $\varepsilon>0$ with the following property. For every   point $x\in M$, every  neighborhood $V$ of  $x$, every symplectic linear isomorphism $A\colon T_xM\to T_{f(x)}M$ with $\|A-Df(x)\|<\varepsilon$, there exists $g\in \cU$ such that:
\begin{itemize}
\item $g(x)=f(x)$ and $g$ coincides with $f$ in the complement of $V$;
\item $Dg(x)=A$.
\end{itemize}
\end{lemm}

\begin{lemm}[Linearizing perturbation]\label{l.symp}
Given a symplectomorphism $f$, a neighborhood $\cU\subset \Symp^1_\omega(M)$ of $f$,
a periodic point $x$ of $f$ of period $n$ and a neighborhood $V$ of the orbit of $x$,
there exists $g\in \cU$ such that:
\begin{itemize}
\item $g$ coincides with $f$ on the orbit of $x$ and in the complement of $V$; in particular $x$ is periodic for $g$;
\item there is a neighborhood $V_0$ of $x$ and   a chart $\psi\colon V_0\to \RR^{2d}$,  such that \begin{itemize}
\item $\psi(x)=0\in\RR^{2d}$,
\item $\psi_*(\omega)$ is the canonical symplectic form on $\RR^{2d}$,
\item the expression of $g^n$ in this chart (i.e. the local symplectomorphism
$\psi\circ g^n\circ \psi^{-1}$) is the symplectic linear map $D\psi(x)\circ Df^n(x)\circ D\psi(x)^{-1}$.
\end{itemize}
\end{itemize}
\end{lemm}

\noindent
\begin{proof}[Proof of Proposition~\ref{p.symp}]
According to Lemmas~\ref{l.symp} and \ref{l.franks} for every $f\in \cO$, any neighborhood $\cU$ of $f$ and any robustly totally elliptic point $x$ of $f$, there exists $g$ in $\cU$ agreeing with $f$ on the orbit of $x$,
and such that:
\begin{itemize}
\item  $Dg^n(x)$ (where $n$ is the period of $x$) is a totally elliptic matrix whose eigenvalues $\lambda_i$ have the form $e^{\alpha_i 2i\pi}$ with $\alpha_i=\frac{p_i}{q_i}$, $p_i,q_i\in\ZZ$, $p_i\wedge q_i=1$;
\item there is a neighborhood $V$ of $x$ and  a symplectic chart $\psi\colon V\to\RR^{2d}$ such that  $\psi(x)=0\in\RR^{2d}$ and the expression of $g^n$  in this chart is $Dg^n(x)$.
\end{itemize}

Let $m$ be the smallest common multiple of the integers $q_i$.
Then $g^{nm}$ is the identity map in a neighborhood of $x$.
It follows that  $g$ admits a periodic ball such that the return map is the identity map, concluding the proof.
\end{proof}

\noindent
\begin{proof}[End of the proof of Theorem~\ref{t.symp}]
Now the proof of Theorem~\ref{t.symp} is identical to the proof of Theorem~\ref{t.wild}: consider a diffeomorphisms $f\in\cD$ and a periodic ball $D$ on which the first return map is the identity map. Shrinking  $D$ if necessary,
we can assume that there is a symplectic chart $\psi\colon D\to \RR^{2d}$
inducing a symplectomorphism from $D$ to the standard ball $\DD^{2d}\subset\RR^{2d}$.

Let $\varphi$ be  a symplectic diffeomorphism of $\RR^{2d}$ that is the identity map on the complement of $\DD^{2d}$.
Let $g_0$ be the diffeomorphism of $M$ that is the identity map on the complement of $D$ and that is $\psi^{-1}\varphi\psi$ on $D$, and let $g_n=f^ng_0 f^{-n}$. Then the diffeomophism $g_\varphi$ coinciding with the identity map in the complement of the $f$-orbit of $D$ and with $g_i$ on $f^i(D)$ is a symplectomorphism of $M$ commuting with $f$.
\end{proof}

\vspace{10pt}

\noindent \textbf{Christian Bonatti (bonatti@u-bourgogne.fr)}\\
\noindent \textbf{Gioia M. Vago (vago@u-bourgogne.fr)}\\
\noindent  CNRS - Institut de Math\'ematiques de Bourgogne, UMR 5584\\
\noindent  Universit\'e de Bourgogne - BP 47 870\\
\noindent  21078 Dijon Cedex, France\\
\vspace{10pt}

\noindent \textbf{Sylvain Crovisier (crovisie@math.univ-paris13.fr)}\\
\noindent CNRS - Laboratoire Analyse, G\'eom\'etrie et Applications, UMR 7539,\\
\noindent Institut Galil\'ee, Universit\'e Paris 13, Avenue J.-B. Cl\'ement,\\
\noindent 93430 Villetaneuse, France\\
\vspace{10pt}

\noindent \textbf{Amie Wilkinson (wilkinso@math.northwestern.edu)}\\
\noindent Department of Mathematics, Northwestern University\\
\noindent 2033 Sheridan Road \\
\noindent Evanston, IL 60208-2730,  USA

\end{document}